\documentclass{article}
\usepackage{graphicx} % Required for inserting images
\usepackage[letterpaper,top=2cm,bottom=2cm,left=3cm,right=3cm,marginparwidth=1.75cm]{geometry}
\usepackage{amsmath,amssymb,amsthm}
\usepackage[colorlinks=true, allcolors=blue, hypertexnames=false]{hyperref}
\usepackage{comment}
\usepackage{commath}
\usepackage{cleveref}
\usepackage{algorithm}
\usepackage{algpseudocode}

\usepackage{tikz}
\usetikzlibrary{arrows.meta,positioning}

\newtheorem{thm}{Theorem}[section]
\newtheorem{proposition}{Proposition}[section]
\newtheorem{corollary}{Corollary}[section]

\newtheorem{definition}{Definition}[section]

\newtheorem{lemma}{Lemma}[section]

\newtheorem{rmk}{Remark}[section]
\newtheorem{assumption}{Assumption}[section]

\numberwithin{equation}{section}

\DeclareMathOperator{\supp}{supp}

\def\cP{\mathcal{P}}

\DeclareMathOperator*{\argmin}{arg\,min}

\title{Kullback--Leibler Mirror-Prox for Measure-Valued Variational Inequalities and Mean-Field Equilibria}
\author{Erhan Bayraktar\thanks{Department of Mathematics, University of Michigan. E-mail addresses: \href{mailto:erhan@umich.edu}{\nolinkurl{erhan@umich.edu}}, \href{mailto:iekren@umich.edu}{\nolinkurl{iekren@umich.edu}}, and \href{mailto:ziqingzh@umich.edu}{\nolinkurl{ziqingzh@umich.edu}}.}\textsuperscript{,\,}\thanks{E. Bayraktar is supported in part by NSF Grant DMS-2602036 and in part by the Susan M. Smith Professorship.}
\and Ibrahim Ekren\footnotemark[2]\textsuperscript{,\,}\thanks{I. Ekren is supported in part by NSF Grant DMS-2406240.}
\and Lu Vy\thanks{Department of Operations Research and Financial Engineering, Princeton University. E-mail address: \href{mailto:lv4809@princeton.edu}{\nolinkurl{lv4809@princeton.edu}}.}
\and Ziqing Zhang\footnotemark[2]}

\hypersetup{
  pdftitle={Kullback–Leibler Mirror-Prox for Measure-Valued Variational Inequalities and Mean-Field Equilibria},
  pdfauthor={Erhan Bayraktar, Ibrahim Ekren, Lu Vy, and Ziqing Zhang},
  pdfsubject={Static mean field games and measure-valued variational inequalities},
pdfkeywords={mean field games, variational inequalities, extragradient method, relative entropy, Lasry--Lions monotonicity}
}
\date{August 2026}

\begin{document}

\maketitle

\begin{abstract}
We study the computation of static mean-field equilibria on a compact state space by formulating the equilibrium condition as a variational inequality over probability measures. We propose an entropic variant of Korpelevich's extragradient algorithm---the Kullback--Leibler Mirror-Prox method---in which Euclidean projections are replaced by relative-entropy proximal steps. Each half-step is therefore an explicit exponential reweighting of the current measure, implemented on a finite state-space discretization. Under Lasry--Lions monotonicity and continuity assumptions, we prove convergence of mesh-refined ergodic averages and obtain finite-iteration Minty-residual and approximate-equilibrium bounds that jointly quantify iteration and discretization errors. Under strong monotonicity, we derive metric convergence rates for the last, best, and averaged iterates. We also develop a KL-type Tikhonov regularization that selects the equilibrium minimizing relative entropy with respect to a reference measure. The framework applies to potential and nonpotential cost operators and does not require differentiability or convexity of the cost in the individual state.
\end{abstract}

\section{Introduction}

Mean field games (MFGs), developed independently by Lasry and Lions
\cite{lasry2007mean} and Huang, Malham\'e, and Caines
\cite{huang2006large}, provide a framework for studying strategic
interactions in large populations. In the mean field limit, the influence of
other agents is summarized by a population distribution, and an
equilibrium requires consistency between this distribution and the optimal
behavior that it induces.

A substantial literature on computational methods for mean field games has
emerged; see \cite{achdou2020mean,lauriere2021numerical,lauriere2022learning}
for surveys. Classical numerical approaches exploit the PDE characterization
of MFGs arising from dynamic optimal control problems. In both stationary and
time-dependent settings, the associated Hamilton--Jacobi and Fokker--Planck
equations can be discretized and solved using, for example, finite-difference
schemes
\cite{achdou2010mean,achdou2013mean},
policy iteration \cite{cacace2021policy}, proximal methods
\cite{briceno2018proximal}, or splitting methods \cite{liu2021computational}. 
A complementary
probabilistic approach represents dynamic MFGs through McKean--Vlasov
forward--backward stochastic differential equations and develops numerical
schemes based on simulation and stochastic approximation
\cite{angiuli2019cemracs}.
Optimization and learning approaches, including fictitious play, conditional-gradient methods,
dual averaging, and online mirror descent, have also been developed under
various structural assumptions such as potentiality or monotonicity
\cite{cardaliaguet2017learning,perrin2020fictitious,lavigne2023generalized,
hadikhanloo2017learning,hadikhanloo2022learning,perolat2022scaling}.
More recently, methods based on reinforcement learning and machine learning,
including Q-learning and deep reinforcement learning, have been developed to
compute MFG equilibria
\cite{guo2019learning,lauriere2022scalable,magnino2026solving}.

The absence of an intrinsic time component in a static MFG considerably
simplifies its equilibrium structure: each agent makes a one-shot choice
against the population distribution, so the equilibrium can be studied
directly on a space of probability measures without an underlying
state-evolution equation. This makes static MFGs particularly amenable to methods that operate directly at the level of population distributions. 

We adopt this perspective and take the population-level cost operator
\(
\mu \longmapsto J(\mu,\cdot)
\)
as primitive. We formulate the static mean field equilibrium problem directly
as a variational inequality on the space of probability measures and apply
the extragradient principle in the geometry induced by relative entropy. This
leads to a KL Mirror-Prox method that acts directly on population
distributions without requiring a PDE representation, a potential structure,
or the solution of an individual best-response problem at each iteration.

Let $\Omega\subset\mathbb{R}^d$ be a nonempty compact state space, and let
$\mathcal{P}(\Omega)$ denote the set of Borel probability measures on
$\Omega$. We consider a cost functional
\[
    J:\mathcal{P}(\Omega)\times\Omega\to\mathbb{R},
\]
where $J(\mu,x)$ is the cost incurred by an agent choosing state $x$ when the
population distribution is $\mu$. A measure
$\mu^*\in\mathcal{P}(\Omega)$ is a mean field equilibrium (MFE) if
\[
    \mu^*\left(
        \argmin_{x\in\Omega}J(\mu^*,x)
    \right)=1.
\]
Thus, at equilibrium, the population is concentrated on states that are
optimal against that same population distribution.

For a continuous function $f:\Omega\to\mathbb{R}$ and a finite signed measure
$\xi$, write
\[
    \langle f,\xi\rangle
    :=
    \int_\Omega f(x)\,\xi(dx).
\]
Under appropriate integrability
conditions, the equilibrium condition is
equivalent to the following variational inequality (VI):
\begin{equation}\label{eq:MFE-VI}
    \left\langle
        J(\mu^*,\cdot),\eta-\mu^*
    \right\rangle
    \geq 0
    \qquad
    \forall \eta\in\mathcal{P}(\Omega).
\end{equation}
Indeed, \eqref{eq:MFE-VI} states that $\mu^*$ minimizes the linear functional
\[
    \eta\longmapsto
    \left\langle J(\mu^*,\cdot),\eta\right\rangle
\]
over $\mathcal{P}(\Omega)$, which holds precisely when $\mu^*$ is supported on
the minimizers of $J(\mu^*,\cdot)$.
Hence the static MFE problem can be viewed as a variational inequality on the
space of probability measures.

This formulation directly parallels the classical Stampacchia variational
inequality. Given a closed convex set $Z\subseteq\mathbb{R}^d$ and an operator
$F:Z\to\mathbb{R}^d$, the latter asks for $z_*\in Z$ such that
\begin{equation}\label{eq:VI-Korpelevich}
    \langle F(z_*),z-z_*\rangle
    \geq 0
    \qquad
    \forall z\in Z.
\end{equation}
Its Minty counterpart is
\begin{equation}\label{eq:classical-MVI}
    \langle F(z),z-z_*\rangle
    \geq 0
    \qquad
    \forall z\in Z.
\end{equation}
Monotonicity of $F$ implies that every solution of
\eqref{eq:VI-Korpelevich} satisfies \eqref{eq:classical-MVI}, while
continuity gives the reverse implication. Thus, for continuous monotone
operators, the VI and MVI solution sets coincide.

Guided by this classical correspondence, we associate the following
measure-space Minty variational inequality with \eqref{eq:MFE-VI}:
\begin{equation}\label{eq:MFE-MVI}
    \left\langle
        J(\eta,\cdot),\eta-\mu^*
    \right\rangle
    \geq 0
    \qquad
    \forall \eta\in\mathcal{P}(\Omega).
\end{equation}
If the cost operator is monotone in the sense that
\[
    \left\langle
        J(\mu,\cdot)-J(\nu,\cdot),\mu-\nu
    \right\rangle
    \geq 0
    \qquad
    \forall \mu,\nu\in\mathcal{P}(\Omega),
\]
then every solution of \eqref{eq:MFE-VI} satisfies
\eqref{eq:MFE-MVI}. Conversely, an appropriate continuity assumption on
$\mu\mapsto J(\mu,\cdot)$ implies that every solution of
\eqref{eq:MFE-MVI} satisfies \eqref{eq:MFE-VI}. Consequently, under
Lasry--Lions monotonicity and suitable continuity assumptions, the MFE condition, the VI
\eqref{eq:MFE-VI}, and the MVI \eqref{eq:MFE-MVI} are equivalent. Precise
statements are given in Subsection~\ref{subsec:eps-MFE}.

The classical VI viewpoint also suggests a natural computational approach.
A fundamental algorithm for solving finite-dimensional monotone VIs is
Korpelevich's extragradient method \cite{korpelevich1976extragradient}. Given
$z_k\in Z$, it performs the predictor--corrector steps
\begin{equation}\label{eq:korpelevich}
\begin{aligned}
    \bar z_k
    &:=
    P_Z\bigl(z_k-\lambda F(z_k)\bigr),\\
    z_{k+1}
    &:=
    P_Z\bigl(z_k-\lambda F(\bar z_k)\bigr),
\end{aligned}
\end{equation}
where \(P_Z:\mathbb{R}^d\to Z\) denotes the Euclidean projection onto the closed convex set \(Z\). The first operator evaluation produces a prediction $\bar z_k$, while the
second evaluates the operator at the predicted point and corrects the
original iterate $z_k$. When $F$ is monotone and $L$-Lipschitz continuous,
the solution set is nonempty, and $\lambda\in(0,1/L)$, classical
extragradient theory ensures convergence to a solution of
\eqref{eq:VI-Korpelevich}.

Equivalently, \eqref{eq:korpelevich} can be written in proximal form:
\[
    \bar z_k
    \in
    \argmin_{z\in Z}
    \left\{
        \lambda\langle F(z_k),z\rangle
        +\frac{1}{2}\lVert z-z_k\rVert^2
    \right\},
\]
and
\[
    z_{k+1}
    \in
    \argmin_{z\in Z}
    \left\{
        \lambda\langle F(\bar z_k),z\rangle
        +\frac{1}{2}\lVert z-z_k\rVert^2
    \right\}.
\]
This formulation shows that squared Euclidean distance is not intrinsic to
the extragradient construction. It can be replaced by a Bregman divergence
adapted to the geometry of the feasible set, as in mirror-prox methods
\cite{nemirovski2004prox}.

This suggests transporting the extragradient principle to the computation of MFEs on the space of probability measures. A natural choice of Bregman divergence on this space is the Kullback--Leibler divergence
\[
    D_{\mathrm{KL}}(\eta\Vert\mu)
    :=
    \begin{cases}
        \displaystyle
        \int_\Omega
        \log\left(\frac{d\eta}{d\mu}\right)d\eta,
        & \eta\ll\mu,\\[1.2ex]
        +\infty,
        & \text{otherwise}.
    \end{cases}
\]
Given a full-support initial distribution
$\mu_0\in\mathcal{P}(\Omega)$, we therefore introduce the
\emph{KL Mirror-Prox method}
\begin{equation*}
\begin{aligned}
    \nu_k
    &\in
    \argmin_{\eta\in\mathcal{P}(\Omega)}
    \left\{
        \lambda
        \left\langle J(\mu_k,\cdot),\eta\right\rangle
        +
        D_{\mathrm{KL}}(\eta\Vert\mu_k)
    \right\},\\
    \mu_{k+1}
    &\in
    \argmin_{\eta\in\mathcal{P}(\Omega)}
    \left\{
        \lambda
        \left\langle J(\nu_k,\cdot),\eta\right\rangle
        +
        D_{\mathrm{KL}}(\eta\Vert\mu_k)
    \right\}.
\end{aligned}
\end{equation*}

The KL divergence is the Bregman divergence generated by negative entropy and is therefore the natural entropic analogue of the Euclidean geometry underlying the classical extragradient method. It offers two further advantages.

First, it is particularly well suited to $\mathcal{P}(\Omega)$: the probability constraints are enforced automatically, and the proximal steps reduce to exponential tilts without projections or inner optimization routines. In particular, when $\Omega$ is finite, each iteration requires only two evaluations of the cost operator and two normalized exponential updates. For a general compact state space, the same updates can be implemented on a finite approximation of $\Omega$. The precise algorithmic formulation is given in Section~\ref{sec:algo}. 

Second, the KL geometry is compatible with the linear duality pairing between functions and signed measures. This is the same pairing that underlies both the MFE variational inequality and Lasry--Lions monotonicity. The latter is the standard monotonicity condition in mean field game theory and, under an appropriate strictness assumption, yields uniqueness of the equilibrium.

\paragraph{Contributions.}
For a general compact state space \(\Omega\), we formulate the static mean
field equilibrium problem as a variational inequality on
\(\mathcal P(\Omega)\) and propose a KL Mirror-Prox method for its computation.
We implement the resulting measure-valued algorithm through finite-dimensional
approximations obtained by discretizing the state space. This framework
provides a unified variational-inequality perspective for analyzing both
algorithmic convergence and state-space discretization error under
Lasry--Lions monotonicity, and applies to both potential and nonpotential cost
operators.

Our first contribution is a convergence theory for general compact state
spaces, together with a systematic connection between variational residuals
and approximate equilibria. Under Lasry--Lions monotonicity and suitable
continuity assumptions, we show that weak limit points of the mesh-refined
ergodic averages are mean field equilibria and derive, for the averaged
iterates, a quantitative Minty residual estimate valid for a finite mesh and
a finite iteration count. We further introduce the notions of approximate MFE,
variational-inequality (VI), and Minty variational-inequality (MVI) solutions
and establish the implications among their exact and approximate forms.
Combining these relations with the Minty residual estimate yields explicit
approximate-MFE guarantees for the averaged iterates, with errors depending
jointly on the mesh size \(h\) and the iteration count \(K\). In particular,
these bounds provide explicit choices of \(h\) and \(K\) sufficient to
achieve any prescribed approximate-MFE accuracy on a general compact state
space.

Our second contribution consists of explicit metric convergence rates for the last iterate, the best iterate, and the averaged iterate under strong Lasry--Lions monotonicity. These estimates quantify the optimization and discretization errors jointly and, in turn, yield explicit choices of the mesh size \(h\) and the number of iterations \(K\) required to attain any prescribed metric accuracy.

Our third contribution is an equilibrium-selection mechanism based on KL-type Tikhonov regularization. Given a full-support reference measure \(\rho\), we show that, whenever the minimum is attained uniquely and is finite, the vanishing-penalty limit selects the mean field equilibrium minimizing
\(
D_{\mathrm{KL}}(\cdot\|\rho)
\)
over the equilibrium set.

Finally, we complement the theoretical results with numerical experiments illustrating the convergence behavior.

\subsection{Related Work}\label{sec:related-work}
Our work connects four strands of the literature: static MFGs and equilibrium dynamics, mirror descent and proximal methods in monotone MFGs, extragradient and two-step variational-inequality methods, and entropy-based regularization and equilibrium selection. We review these connections separately, emphasizing
how the population-level KL Mirror-Prox formulation and its quantitative
iteration--discretization analysis differ from existing approaches.

\paragraph{Static MFGs and equilibrium dynamics.}
Several recent works address static mean field equilibria directly.
Hynd \cite{hynd2023evolution} studies
continuous-time mixed-strategy dynamics for monotone static MFGs and proves
convergence of their Ces\`aro averages to mean field equilibria. Yardim,
Cayci, and He \cite{yardim2025variational} formulate finite-action static MFGs
as variational inequalities and develop independent learning algorithms with
finite-sample guarantees. Tangpi and Touzi
\cite{tangpi2025particle} approximate static mean field equilibria through
McKean--Vlasov Langevin dynamics and associated particle systems under
Lasry--Lions or displacement monotonicity.
The quantitative equilibrium-approximation regimes considered there impose
either positive displacement monotonicity or Lasry--Lions monotonicity
together with uniform convexity in the individual state. Although this
consequence is not emphasized there, both regimes yield Dirac MFEs in their
setting. Under positive displacement monotonicity, applying the condition to
the product coupling of an equilibrium measure with itself forces its support
to be a singleton. Under uniform convexity in the individual state, the
equilibrium cost has a unique minimizer and hence again supports only a Dirac
mass. Our setting imposes neither positive displacement monotonicity nor convexity of
\(x \mapsto J(\mu,x)\) and therefore allows genuinely non-Dirac mean field
equilibria with multiple actions in their support. Methodologically, we
instead apply a KL Mirror-Prox iteration directly to the population-level
variational inequality.

\paragraph{Mirror descent and proximal methods in monotone MFGs.}
Mirror-type methods have been studied extensively for MFGs and closely related
anonymous nonatomic games. Hadikhanloo \cite{hadikhanloo2017learning}
analyzes online mirror descent (OMD) for anonymous nonatomic games, including
first-order MFGs, under monotonicity and convexity of the individual cost in
the player's action. Hadikhanloo, Laraki, Mertikopoulos, and Sorin
\cite{hadikhanloo2022learning} study fictitious play and
dual-averaging/FTRL dynamics for finite-action monotone games of this type.
P\'erolat et al. \cite{perolat2022scaling} apply OMD to finite-state,
finite-action dynamic MFGs through policy updates based on \(Q\)-values,
proving qualitative convergence of continuous-time OMD under a strict
weak-monotonicity condition. More recently, Isobe, Abe, and Ariu
\cite{isobe2026last} establish last-iterate convergence under non-strict
monotonicity using a KL proximal-point scheme in which each outer iteration
requires solving a KL-regularized MFG subproblem.

Our main algorithmic distinction is the extragradient correction. Although the
corrected iterates admit an entropic-FTRL representation in terms of the
cumulative corrected costs \(J(\nu_t,\cdot)\), KL Mirror-Prox distinguishes
itself through the intermediate predictor used to generate these costs and yields
finite-iteration guarantees for the Minty residual under non-strict
Lasry--Lions monotonicity. In addition, our method acts directly on the
population measure \(\mu\in\mathcal P(\Omega)\) through evaluations of the cost
operator \(J(\mu,\cdot)\); hence, unlike action-space OMD as in Hadikhanloo
\cite{hadikhanloo2017learning}, it does not require convexity of
\(x\mapsto J(\mu,x)\). Our analysis jointly quantifies state-space
discretization and iteration errors through the mesh size \(h\) and iteration
count \(K\). This joint analysis differs from both learning analyses for a
fixed finite model and discrete-to-continuous consistency results such as
Hadikhanloo and Silva \cite{hadikhanloo2019finite}, which prove convergence of refined
finite MFG equilibria to a continuous first-order MFG but do not provide a
quantitative mesh-error bound or account for terminating the equilibrium
algorithm after finitely many iterations.

\paragraph{Extragradient and two-step mirror methods for MFGs.}
Extragradient and related two-step mirror methods have recently been
applied to MFG problems. Meynard
\cite{meynard2026extragradient} formulates mean field games of controls and
associated mean-field FBSDEs as monotone variational inequalities in a Hilbert
space and develops an extragradient method with quantitative convergence under sufficiently
strong displacement-monotonicity assumptions. While preparing the present
version, we became aware of the June 2026 preprint by Abdulaziz, Ashrafyan,
Gevorgyan, and Gomes \cite{abdulaziz2026bregman}. They formulate a
low-order regularization of a stationary MFG PDE system as a variational
inequality on a product of Lebesgue and Sobolev spaces and introduce a
mixed-exponent Bregman geometry adapted to that Banach-space formulation.
Their two mirror steps reuse the same operator evaluation at the current
iterate rather than reevaluating the operator at the intermediate predictor.
Their convergence theory relies on a quantitative Lasry--Lions-type
monotonicity condition on the density coupling and establishes strong
asymptotic convergence for each fixed regularization parameter.

The settings and guarantees differ in several respects. We study a static,
one-shot MFG directly on \(\mathcal P(\Omega)\) for a general compact state
space, rather than a stationary MFG system for a density--value-function
pair. Our KL Mirror-Prox corrector reevaluates the cost operator at the
intermediate predictor, and we obtain explicit finite-iteration and finite-mesh
guarantees under non-strict Lasry--Lions monotonicity. We also quantify metric convergence under strong
monotonicity and develop a reference-dependent KL equilibrium-selection
mechanism for the original, unregularized static MFG.

\paragraph{Entropy-based regularization and equilibrium selection.}
Entropy and relative entropy have been used in MFGs primarily to regularize
algorithms or control problems. Guo, Xu, and Zariphopoulou
\cite{guo2022entropy} study entropy regularization in finite-horizon MFG
learning, where it promotes exploration and can improve stability and
convergence. Related reinforcement-learning approaches with entropy
regularization use Boltzmann-type policies to smooth the best-response map and
facilitate convergence to approximate MFG equilibria
\cite{cui2021approximately}. In these works, entropy regularization is built
into the learning or control problem itself, yielding a regularized MFG with
modified best responses and equilibria.

By contrast, we use KL regularization as a vanishing perturbation to select an equilibrium of the original, unregularized problem. This is closer in spirit to Tikhonov regularization for monotone variational inequalities, which is used to regularize merely monotone problems and obtain stable iterative
procedures; see, e.g., \cite{koshal2012regularized}. Equilibrium selection in MFGs has also been studied through other vanishing perturbations. For example, Cecchin and Delarue \cite{cecchin2022selection} establish a selection principle based on vanishing common noise for potential finite-state MFGs. Our mechanism differs in that it selects the equilibrium through an explicit reference-dependent optimization criterion over the equilibrium set.

\paragraph{Organization of the paper.}
The remainder of the paper is organized as follows.
Section~\ref{sec:algo} introduces the KL Mirror-Prox method and its
finite-dimensional implementation on discrete state-space approximations.
Section~\ref{sec:LL-mono} develops the ergodic convergence theory under
Lasry--Lions monotonicity and establishes the implications among the
corresponding approximate MFE, VI, and MVI notions.
Section~\ref{sec:strong-LL-mono} establishes finite-iteration metric
convergence rates under strong Lasry--Lions monotonicity.
Section~\ref{sec:MFE-selection} develops an equilibrium-selection mechanism
based on a KL-type Tikhonov penalty. Finally, Section~\ref{sec:experiment}
presents numerical experiments. We next introduce notation used throughout
the paper.

\paragraph{Notation.}
Let \(\Omega\subset\mathbb R^d\) be compact. We denote by \(\mathcal P(\Omega)\) the set of Borel probability measures on \(\Omega\). For two probability measures \(\mu\) and \(\nu\), their
\(L^1\)-distance is
\[
    \lVert\mu-\nu\rVert_1
    :=
    \int_{\Omega}
        \left|
            \frac{d\mu}{d\xi}
            -
            \frac{d\nu}{d\xi}
        \right|
    d\xi,
\]
where \(\xi\) is any measure dominating both \(\mu\) and \(\nu\).
Equivalently,
\[
    \lVert\mu-\nu\rVert_1
    =
    \int_{\Omega}|d\mu-d\nu|.
\]
With this convention, \(\lVert\mu-\nu\rVert_1\) is twice the usual total variation distance.

The Kullback--Leibler divergence of \(\eta\) from \(\mu\) is defined
by
\[
    D_{\mathrm{KL}}(\eta\Vert\mu)
    :=
    \begin{cases}
        \displaystyle
        \int_{\Omega}
            \log\left(\frac{d\eta}{d\mu}\right)d\eta,
        & \eta\ll\mu,\\[1.2ex]
        +\infty,
        & \text{otherwise}.
    \end{cases}
\]
For \(\mu,\nu\in\mathcal P(\Omega)\), their
\(1\)-Wasserstein distance is
\[
    W_1(\mu,\nu)
    :=
    \inf_{\pi\in\Pi(\mu,\nu)}
    \int_{\Omega\times\Omega}
        |x-y|\,\pi(dx,dy),
\]
where \(\Pi(\mu,\nu)\) denotes the set of couplings of \(\mu\) and \(\nu\).
Finally, for an integrable function \(f\), we recall the shorthand
\[
    \langle f,\mu\rangle
    :=
    \int_{\Omega}f(x)\,\mu(dx).
\]

\section{KL Mirror-Prox and Finite-Dimensional Implementation}\label{sec:algo}

Let $\Omega\subset\mathbb{R}^d$ be compact. Given a stepsize
$\lambda>0$ and an initial measure
$\mu_0\in\mathcal{P}(\Omega)$, we define the KL Mirror-Prox iteration by
\begin{equation*}
\begin{aligned}
    \nu_k
    &\in
    \argmin_{\eta\in\mathcal{P}(\Omega)}
    \left\{
        \lambda
        \left\langle J(\mu_k,\cdot),\eta\right\rangle
        +
        D_{\mathrm{KL}}(\eta\Vert\mu_k)
    \right\},\\
    \mu_{k+1}
    &\in
    \argmin_{\eta\in\mathcal{P}(\Omega)}
    \left\{
        \lambda
        \left\langle J(\nu_k,\cdot),\eta\right\rangle
        +
        D_{\mathrm{KL}}(\eta\Vert\mu_k)
    \right\}.
\end{aligned}
\end{equation*}
The first subproblem is the entropic prediction step, while the second
evaluates the cost at the predictor $\nu_k$ and corrects the original
iterate $\mu_k$. Both subproblems are regularized relative to $\mu_k$,
mirroring the fact that both proximal steps of the classical
Mirror-Prox method are centered at the current iterate.

\begin{comment}
This two-step correction should not be confused with fictitious play. In the classical MFG fictitious-play scheme of Cardaliaguet and Hadikhanloo \cite{cardaliaguet2017learning}, each stage computes an exact best response to the current empirical average and then updates that average with a vanishing weight. The convergence result established there is for potential MFGs. By contrast, KL Mirror-Prox requires only two evaluations of the cost operator and two explicit entropic reweightings per iteration, uses a constant stepsize, and provides nonasymptotic residual bounds under Lasry--Lions monotonicity and the stated regularity assumptions, without requiring a potential structure.
\end{comment}

Whenever the corresponding normalizing constants are finite and positive,
the minimizers are uniquely given by the exponential tilts
\begin{equation}\label{eq:KL-MP-Gibbs}
\begin{aligned}
    \frac{d\nu_k}{d\mu_k}(x)
    &=
    \frac{
        \exp\bigl(-\lambda J(\mu_k,x)\bigr)
    }{
        \displaystyle
        \int_{\Omega}
            \exp\bigl(-\lambda J(\mu_k,z)\bigr)\,\mu_k(dz)
    },\\
    \frac{d\mu_{k+1}}{d\mu_k}(x)
    &=
    \frac{
        \exp\bigl(-\lambda J(\nu_k,x)\bigr)
    }{
        \displaystyle
        \int_{\Omega}
            \exp\bigl(-\lambda J(\nu_k,z)\bigr)\,\mu_k(dz)
    }.
\end{aligned}
\end{equation}

{
KL Mirror-Prox is related to both fictitious play and follow-the-regularized-leader (FTRL), but differs from them in important ways. In the classical MFG fictitious-play scheme of Cardaliaguet and Hadikhanloo \cite{cardaliaguet2017learning}, each iteration computes an exact best response to an empirical average of past population measures and then updates this average with a vanishing weight. The convergence result established there relies on a potential structure. By contrast, KL Mirror-Prox neither solves exact best-response problems nor requires the MFG to be potential; each iteration instead uses two evaluations of the cost operator and two explicit entropic reweightings.

From this perspective, KL Mirror-Prox is more closely related to the dual-averaging/FTRL dynamics studied by Hadikhanloo et al.\ \cite{hadikhanloo2022learning}. Whereas fictitious play averages past \emph{measures}, FTRL aggregates past \emph{costs} and computes a regularized response to the resulting cumulative cost. Indeed, recursively expanding the corrected update in \eqref{eq:KL-MP-Gibbs} gives
\[
    \mu_k(dx)
    \propto
    \mu_0(dx)
    \exp\!\left(-\lambda\sum_{t=0}^{k-1}J(\nu_t,x)\right),
\]
showing that $\mu_k$ is an entropically regularized response to the cumulative corrected cost evaluations. KL Mirror-Prox can therefore be interpreted as an optimistic (extragradient) modification of entropic FTRL: the evaluation $J(\mu_k,\cdot)$ is used first to form the predictor $\nu_k$, and the corrected evaluation $J(\nu_k,\cdot)$ is then incorporated into the next iterate. This predictor--corrector step is particularly advantageous for monotone problems: under Lasry--Lions monotonicity and suitable Lipschitz regularity of $J$, it permits a constant stepsize and an $O(K^{-1})$ ergodic residual bound. By comparison, the discrete-time dual-averaging analysis of Hadikhanloo et al.\ uses a vanishing learning rate of order $K^{-1/2}$; its regret estimate yields an $O(K^{-1/2})$ ergodic bound for monotone games.}

\paragraph{Finite state space.}
Suppose that $\Omega=\{x_1,\ldots,x_N\}$. Identifying
\[ \mathcal{P}(\Omega)\simeq\Delta_N:= \left\{ p\in\mathbb{R}^N: p_i\ge0,\quad\sum_{i=1}^N p_i=1 \right\}, \]
for $k=0,1,\dots,K$, write
\[
    \mu_k=\sum_{i=1}^N p_{k,i}\delta_{x_i}\in\Delta_N,
    \qquad
    \nu_k=\sum_{i=1}^N q_{k,i}\delta_{x_i}\in\Delta_N.
\]
Then \eqref{eq:KL-MP-Gibbs} reduces to the explicit updates
\begin{equation}\label{eq:finite-KL-MP}
\begin{aligned}
    q_{k,i}
    &=
    \frac{
        p_{k,i}e^{-\lambda J(\mu_k,x_i)}
    }{
        \displaystyle
        \sum_{j=1}^N
        p_{k,j}e^{-\lambda J(\mu_k,x_j)}
    },\\
    p_{k+1,i}
    &=
    \frac{
        p_{k,i}e^{-\lambda J(\nu_k,x_i)}
    }{
        \displaystyle
        \sum_{j=1}^N
        p_{k,j}e^{-\lambda J(\nu_k,x_j)}
    },
    \qquad i=1,\ldots,N.
\end{aligned}
\end{equation}
% Starting from a full-support initialization
% \[
%     p_0\in\Delta_N^\circ
%     :=
%     \left\{
%         p\in\mathbb{R}^N:
%         p_i> 0,\quad
%         \sum_{i=1}^N p_i=1
%     \right\},
% \]
% the iterates remain in $\Delta_N^\circ$, provided the evaluated costs are
% finite. 

Once the two cost vectors have been evaluated, each iteration requires only
componentwise exponentiation, multiplication, and normalization. Thus,
excluding the cost evaluations, its computational cost is $O(N)$ per
iteration, with no projection or inner optimization routine. The choices of the stepsize $\lambda$ and the number of iterations $K$ are
specified later through the convergence theory.

\paragraph{General compact state space.}
Let $h>0$. Let $X_h$ be a finite subset of $\Omega$ and $Q_h:\Omega\to X_h$ be a measurable mesh projection
satisfying
\begin{equation}\label{eq:mesh-projection}
    \sup_{x\in\Omega}
    \lvert Q_h(x)-x\rvert
    \leq h.
\end{equation}
% Write
% \[
%     X_h=\{x_1^h,\ldots,x_{N_h}^h\}.
% \]
Since $X_h\subset\Omega$, every measure in $\mathcal{P}(X_h)$ is naturally
viewed as an element of $\mathcal{P}(\Omega)$.

The KL Mirror-Prox method is then applied to the finite state space $X_h$ exactly as in the finite-state case above.
Its computational cost, excluding cost evaluations, is $O(N_h)$. 
The convergence theory below specifies the choices of the stepsize
\(\lambda\), mesh resolution \(h\), and number of iterations \(K\).

\begin{rmk}
\label{rem:existence-mesh-projections}
The existence of the finite sets $X_h$ and measurable projections $Q_h$ is
automatic when $\Omega$ is a compact metric space. Indeed, compactness
implies total boundedness, so for every $h>0$ there exists a finite
$h$-net
\[
    X_h=\{x_1,\ldots,x_{N_h}\}\subset\Omega
\]
such that
\[
    \Omega
    \subseteq
    \bigcup_{i=1}^{N_h}B(x_i,h).
\]
After fixing an ordering of $X_h$, define $Q_h(x)$ as the first point
$x_i$ satisfying
\(
    x\in B(x_i,h).
\)
Equivalently,
\[
    Q_h^{-1}(\{x_i\})
    =
    \Omega\cap
    \left(
        B(x_i,h)
        \setminus
        \bigcup_{j<i}B(x_j,h)
    \right).
\]
The fibers of $Q_h$ are Borel sets, so $Q_h$ is measurable. Moreover, \eqref{eq:mesh-projection} is satisfied since
\(
    \lvert Q_h(x)-x\rvert<h
\)
for all $x\in\Omega$.
\end{rmk}

\section{Ergodic Convergence under Lasry--Lions Monotonicity}\label{sec:LL-mono}
This section develops the ergodic convergence theory of the KL Mirror-Prox method under Lasry--Lions monotonicity. The analysis proceeds in four stages. We first derive the KL three-point identity and the resulting one-step Mirror-Prox estimate. We then telescope this estimate to obtain an \(O(K^{-1})\) ergodic Minty residual bound. Next, we pass from finite state spaces to a general compact state space through finite-mesh consistency and mesh refinement, identifying weak limit points of the mesh-refined averages as MFEs and deriving quantitative Minty residual bounds. Finally, we establish the relations among approximate MFE, VI, and MVI conditions, which convert these Minty residual bounds into approximate-MFE guarantees with explicit dependence on mesh size \(h\) and iteration count \(K\).

\begin{assumption}[Standing assumption]\label{ass:standing}
Let \(\Omega\) be compact and define
\[
    D_\Omega:=\operatorname{diam}(\Omega) = \sup_{x,y\in\Omega}|x-y|.
\]
Assume that
\(
J(\mu,\cdot)\in C(\Omega)
\)
for every \(\mu\in\mathcal P(\Omega)\).
\end{assumption}
Compactness of the state space plays several roles. It ensures that every continuous cost \(J(\mu,\cdot)\) is bounded and uniformly continuous, that \(\mathcal P(\Omega)\) is compact for the topology of weak convergence, and that \(W_1\) metrizes weak convergence on \(\mathcal P(\Omega)\). Moreover, it allows the comparison between \(W_1\) and total variation. We assume throughout that Assumption~\ref{ass:standing} holds.

We next introduce several additional assumptions. Unlike the standing assumption above, these conditions will be imposed only for the results in which they are explicitly invoked.

\begin{assumption}[Lasry--Lions monotonicity]\label{ass:LL-mono}
For every \(\mu,\nu\in\mathcal P(\Omega)\),
\[
\langle J(\mu,\cdot)-J(\nu,\cdot),\mu-\nu\rangle\ge 0.
\]
\end{assumption}

\begin{assumption}[\(W_1\)-Lipschitz continuity]\label{ass:W1-lip}
There exists \(L_m>0\) such that, for every \(\mu,\nu\in\mathcal P(\Omega)\),
\[
\|J(\mu,\cdot)-J(\nu,\cdot)\|_\infty
\le
L_m W_1(\mu,\nu).
\]
\end{assumption}

\subsection{The KL three-point identity and one-step estimate}\label{subsec:KL-estimate}

Since the Kullback--Leibler divergence is the Bregman divergence associated with negative entropy, its three-point identity yields the fundamental descent estimates underlying mirror descent and Mirror-Prox methods.

In the present setting, the proximal subproblem is posed over the entire probability space \(\mathcal P(\Omega)\), and its unique solution is given by an exponential tilt of the reference measure. Consequently, the usual Bregman first-order inequality strengthens to an exact three-point identity. This identity also preserves finite relative entropy: any comparison measure with finite KL divergence relative to the initial distribution \(\mu_0\) has finite KL divergence relative to every measure generated by the KL Mirror-Prox iteration.

We begin with the precise form of this identity.

\begin{lemma}\label{lem:three-point}
Let \(\rho\in\mathcal P(\Omega)\), \(g\in C(\Omega)\), and \(\lambda>0\). Define \(\zeta\in\mathcal P(\Omega)\) by
\[
\frac{d\zeta}{d\rho}(x)
=
\frac{e^{-\lambda g(x)}}{\int_\Omega e^{-\lambda g(x)}\,\rho(dx)}.
\]
Then \(\zeta\sim\rho\) and it is the unique minimizer of
\[
\inf_{\eta\in\mathcal P(\Omega)}
\left\{
\lambda\langle g,\eta\rangle
+
D_{\mathrm{KL}}(\eta\|\rho)
\right\}.
\]
Moreover, for every \(\eta\in\mathcal P(\Omega)\) such that
\(D_{\mathrm{KL}}(\eta\|\rho)<\infty\), we have
\[
D_{\mathrm{KL}}(\eta\|\zeta)<\infty
\]
and
\[
\lambda\langle g,\zeta-\eta\rangle
=
D_{\mathrm{KL}}(\eta\|\rho)
-
D_{\mathrm{KL}}(\eta\|\zeta)
-
D_{\mathrm{KL}}(\zeta\|\rho).
\]
\end{lemma}

We apply Lemma~\ref{lem:three-point} to each of the two subproblems within a single iteration. Adding the resulting identities, the intermediate term \(\mathrm{KL}(\mu_{k+1}\|\mu_k)\) cancels, while the remaining KL terms provide dissipation. The only additional term is the classical Mirror-Prox mismatch arising from the discrepancy between \(J(\mu_k,\cdot)\) and \(J(\nu_k,\cdot)\). Using Assumption~\ref{ass:W1-lip}, together with Pinsker’s and Young’s inequalities, this term can be absorbed into the KL dissipation, giving the following one-step estimate.

\begin{lemma}\label{lem:one-step-Mirror-Prox}
Suppose Assumption~\ref{ass:W1-lip} holds. 
Let $\lambda>0$ and let \((\mu_k,\nu_k)_{k\ge 0}\) be the KL Mirror-Prox iterates initiated from \(\mu_0\in\mathcal P(\Omega)\). Then, for every \(\eta\in\mathcal P(\Omega)\) such that
\(
D_{\mathrm{KL}}(\eta\|\mu_0)<\infty,
\)
we have, for every \(k\ge 0\),
\[
D_{\mathrm{KL}}(\eta\|\mu_k)<\infty,
\qquad D_{\mathrm{KL}}(\eta\|\nu_k)<\infty,
\]
and
\[
\lambda\langle J(\nu_k,\cdot),\nu_k-\eta\rangle
\le
D_{\mathrm{KL}}(\eta\|\mu_k)
-
D_{\mathrm{KL}}(\eta\|\mu_{k+1})
-
(1-\lambda L_m D_\Omega)
\left[
D_{\mathrm{KL}}(\nu_k\|\mu_k)
+
D_{\mathrm{KL}}(\mu_{k+1}\|\nu_k)
\right].
\]
\end{lemma}

\subsection{Ergodic Minty estimate}
\label{subsec:ergodic-minty}
As discussed in the introduction, under Lasry--Lions monotonicity and appropriate continuity assumptions, the mean field equilibrium condition is equivalent to both the VI and the MVI. This motivates a Minty-based analysis of the averaged iterates as in \cite{nemirovski2004prox}.

Summing the one-step estimate in Lemma~\ref{lem:one-step-Mirror-Prox} over \(k\) yields a telescoping bound. Using Lasry--Lions monotonicity and linearity in the measure argument, we obtain the following Minty residual estimate for the averaged iterates.

\begin{lemma}
\label{lem:ergodic-minty}
Suppose Assumptions~\ref{ass:LL-mono} and~\ref{ass:W1-lip} hold. Let
\begin{equation}
\label{eq:step-size-condition}
0<\lambda<\frac{1}{L_mD_\Omega}
\end{equation}
and let \((\mu_k,\nu_k)_{k\ge 0}\) be the KL Mirror-Prox iterates initialized
from \(\mu_0\in\mathcal P(\Omega)\). Define the averaged measure
\[
\bar\nu_K
:=
\frac1K\sum_{k=0}^{K-1}\nu_k.
\]
Then, for every \(\eta\in\mathcal P(\Omega)\) such that \(D_{\mathrm{KL}}(\eta\|\mu_0)<\infty\),
\[
\langle J(\eta,\cdot),\bar\nu_K-\eta\rangle
\le
\frac{D_{\mathrm{KL}}(\eta\|\mu_0)}{\lambda K}.
\]
\end{lemma}

Because \(\Omega\) is compact, the sequence of averaged measures is relatively compact under weak convergence. Passing the ergodic estimate to a subsequential limit yields the following restricted Minty inequality.

\begin{corollary}\label{cor:cluster-point-mintyVI}
Suppose Assumptions~\ref{ass:LL-mono} and~\ref{ass:W1-lip} hold.
Let \[0<\lambda<\frac{1}{L_m D_\Omega}\]
and let \((\mu_k,\nu_k)_{k\ge 0}\) be the KL Mirror-Prox iterates initialized
from \(\mu_0\in\mathcal P(\Omega)\). Define
\[
\bar\nu_K
:=
\frac1K\sum_{k=0}^{K-1}\nu_k.
\]
Then there exist a subsequence \((\bar\nu_{K_j})_{j\ge 1}\) and a measure
\(\bar\nu\in\mathcal P(\Omega)\) such that
\[
\bar\nu_{K_j}\Rightarrow \bar\nu
\qquad\text{in }\mathcal P(\Omega).
\]
Moreover, every such subsequential limit satisfies
\[
\langle J(\eta,\cdot),\bar\nu-\eta\rangle\le 0
\qquad
\forall \eta\in\mathcal P(\Omega)
\text{ such that }
D_{\mathrm{KL}}(\eta\|\mu_0)<\infty.
\]

\end{corollary}

The restriction $D_{\mathrm{KL}}(\eta\|\mu_0)<\infty$ in both Lemma~\ref{lem:ergodic-minty} and Corollary~\ref{cor:cluster-point-mintyVI} reflects the support-preserving nature of the iterations. Indeed, each iterate is an exponential tilt, so the algorithm cannot create mass outside \(\operatorname{supp}\mu_0\). Consequently, the preceding results only test the averaged iterates against measures that are compatible with the support and absolute-continuity structure of the initialization.

This restriction becomes especially transparent when the state space is finite. If \(\Omega=\{x_1,\ldots,x_n\}\), then \[ D_{\mathrm{KL}}(\eta\|\mu_0)<\infty \quad\Longleftrightarrow\quad \operatorname{supp}\eta \subseteq \operatorname{supp}\mu_0. \]  
Hence, in the finite-state setting, every cluster point \(\bar\nu\) of the averaged iterates satisfies \[ \langle J(\eta,\cdot),\bar\nu-\eta\rangle\leq 0 \] for every redistribution \(\eta\) supported on the set initially explored by \(\mu_0\).
In the mean field equilibrium interpretation, this condition rules out profitable redistributions among all populations representable on that support. In particular, if \(\mu_0\) has full support on \(\Omega\), then the Minty inequality holds against every
\(\eta\in\mathcal P(\Omega)\). Subsection~\ref{subsec:eps-MFE} shows that this implies that \(\bar\nu\) is a mean field equilibrium.

Moreover, when \(\Omega\) is finite and \(\mu_0\) has full support on \(\Omega\), \(D_{\mathrm{KL}}(\eta\|\mu_0)\) can be bounded uniformly over \(\eta\in\mathcal P(\Omega)\). Applying the ergodic estimate in Lemma~\ref{lem:ergodic-minty} with this uniform bound yields a Minty residual estimate that quantifies the error after $K$ iterations and is uniform over the entire \(\mathcal P(\Omega)\).

\begin{corollary}
\label{cor:finite-mesh-approximate-MFE}
Suppose Assumptions~\ref{ass:LL-mono} and~\ref{ass:W1-lip} hold.
Let \(\Omega\) be finite and
\(\mu_0\in\mathcal P(\Omega)\) have full support. Define
\[
    m
    :=
    \min_{x\in \Omega}\mu_0(\{x\})>0.
\]
Let
\[
    0<\lambda<\frac{1}{L_mD_\Omega}.
\]
Let \((\mu_k,\nu_k)_{k\geq0}\) be the KL
Mirror-Prox iterates initialized from \(\mu_0\), and define
\[
    \bar\nu_K
    :=
    \frac1K\sum_{k=0}^{K-1}\nu_k.
\]
Then 
\[
    \bigl\langle
        J(\eta,\cdot),
        \eta-\bar\nu_K
    \bigr\rangle
    \geq
    -\varepsilon_K,
    \qquad
    \forall \eta\in\mathcal P(\Omega),
\]
where
\[
    \varepsilon_K
    :=
    \frac{\log(1/m)}{\lambda K}.
\]
\end{corollary}
 
In the finite-state setting, full support of \(\mu_0\) makes every \(\eta\in\mathcal P(\Omega)\) an admissible comparison measure. Consequently, cluster points of the averaged iterates satisfy the full Minty inequality, and the ergodic estimate yields a residual bound that is uniform over \(\mathcal P(\Omega)\).

For a general compact state space, the main obstacle is that the admissible class of comparison measures is restricted by the condition \(D_{\mathrm{KL}}(\eta\|\mu_0)<\infty\). In the next subsection, we remove this restriction through a mesh-refinement argument, extending both the cluster-point convergence result and the uniform Minty residual estimate to the general compact setting.

\subsection{Finite-mesh consistency and mesh refinement} \label{subsec:finite-mesh-approximation}

% \begin{assumption}[Weak-to-uniform continuity]\label{ass:W1-lip}
%     Assume that \(J:\mathcal P(\Omega)\times\Omega\to\mathbb R\) satisfies
%     the following continuity condition:
%     whenever \(\rho_h\Rightarrow \rho\) weakly in \(\mathcal P(\Omega)\),
%     \[
%     \|J(\rho_h,\cdot)-J(\rho,\cdot)\|_\infty\to0.
%     \]
% \end{assumption}

The previous subsection established convergence and residual estimates for the KL Mirror-Prox method on finite state spaces. We now connect these discrete results to the problem on a general compact state space.

We begin with a finite-time consistency result. Proposition~\ref{prop:fixed-K-consistency} shows that, when the discrete initializations approximate a given initialization on \(\Omega\), the finite-support iterates converge to the corresponding iterates of the KL Mirror-Prox method on \(\mathcal P(\Omega)\) at every fixed iteration horizon. Thus, finite meshes provide a consistent approximation of the algorithm itself.

Here, the number of iterations is kept fixed while the mesh is refined.
For each finite subset \(X_h\subseteq\Omega\), we identify \(\mathcal P(X_h)\) with the subset of \(\mathcal P(\Omega)\) consisting of probability measures supported on \(X_h\). Accordingly, all weak convergence statements are understood in the ambient space \(\mathcal P(\Omega)\).

\begin{proposition}
\label{prop:fixed-K-consistency}
Suppose Assumption~\ref{ass:W1-lip} holds. 
Let \(\mu_0\in\mathcal P(\Omega)\). Let
\((X_h)_{h>0}\) be a family of finite subsets of \(\Omega\). Let \(\mu_{0,h}\in\mathcal P(X_h)\) satisfy
\[
\mu_{0,h}\Rightarrow \mu_0
\qquad\text{in }\mathcal P(\Omega)\text{ as }h\to 0.
\]

Let \(\lambda>0\)
and let \((\mu_k,\nu_k)_{k\ge 0}\) and \((\mu_{k,h},\nu_{k,h})_{k\ge 0}\) be the KL Mirror-Prox iterates initialized
from \(\mu_0\) and \(\mu_{0,h}\), respectively.
Then, for every fixed \(k\in\mathbb N\),
\[
\mu_{k,h}\Rightarrow \mu_k,
\qquad
\nu_{k,h}\Rightarrow \nu_k.
\]
Consequently, for every \(K\in\mathbb N\),
\[
\bar\nu_{K,h}
:=
\frac1K\sum_{k=0}^{K-1}\nu_{k,h}
\Rightarrow
\bar\nu_K
:=
\frac1K\sum_{k=0}^{K-1}\nu_k.
\]

\end{proposition}

Approximating the mean field equilibrium requires the iteration count to diverge and the mesh size to vanish. Increasing the iteration count drives the finite-mesh Minty residual to zero, whereas mesh refinement enlarges the class of admissible comparison measures until it becomes dense in \(\mathcal P(\Omega)\). Theorem~\ref{thm:mesh-refinement-MFE} provides a condition coupling these two limits and shows that every subsequential limit \(\bar\nu\) of the ergodic averages solves \eqref{eq:MFE-MVI}. Subsection~\ref{subsec:eps-MFE} then shows that every such \(\bar\nu\) is a mean field equilibrium.

\begin{thm}
\label{thm:mesh-refinement-MFE}
Suppose Assumptions~\ref{ass:LL-mono} and~\ref{ass:W1-lip} hold.
Let
\((X_h)_{h>0}\) be a family of finite subsets of \(\Omega\), and, for each $h>0$, let $Q_h:\Omega\to X_h$ be a measurable mesh projection such that
\[
    \sup_{x\in\Omega} |Q_h(x)-x|\le h.
\]

For each \(h>0\), let \(\mu_{0,h}\in\mathcal P(X_h)\) be an arbitrary
full-support initialization:
\[
\mu_{0,h}(\{x\})>0
\qquad
\forall x\in X_h.
\]
Define its minimal mesh mass by
\[
m_h
:=
\min_{x\in X_h}\mu_{0,h}(\{x\}).
\]
Let \[0<\lambda<\frac{1}{L_m D_\Omega}.\] For each \(h\), let \((\mu_{k,h},\nu_{k,h})_{k\ge 0}\) be the KL Mirror-Prox iterates initialized
from \(\mu_{0,h}\).
Choose \(K_h\to\infty\) such that
\[
\frac{\log(1/m_h)}{K_h}\to 0.
\]
Define
\[
\bar\nu_h
:=
\bar\nu_{K_h,h}
:=
\frac1{K_h}\sum_{k=0}^{K_h-1}\nu_{k,h}.
\]
Then every weak limit point of \((\bar\nu_h)_{h>0}\) is a solution of \eqref{eq:MFE-MVI}. More precisely, if along some sequence
\(h_j\downarrow0\),
\[
\bar\nu_{h_j}\Rightarrow\bar\nu,
\]
then
\[
\langle J(\eta,\cdot),\bar\nu-\eta\rangle\le0,
\qquad
\forall \eta\in\mathcal P(\Omega).
\]
% and $\bar\nu$ is hence a mean field equilibrium.
% Moreover, if it is the unique mean field equilibrium, then the whole sequence converges:
% \[
% \bar\nu_h\Rightarrow\bar\nu.
% \]
\end{thm}

Note that Theorem~\ref{thm:mesh-refinement-MFE} does not require the initializations
\(\mu_{0,h}\) to converge weakly to a limiting measure \(\mu_0\). For this result, the initialization enters only through the effective support
and the entropy scale. More precisely, the relevant assumptions are that \(\operatorname{supp}\mu_{0,h}\) becomes dense in \(\Omega\) as $h\to 0$,
and that the entropy error is negligible on the ergodic timescale:
\[
\frac{\log\left(1/\min_{x\in \operatorname{supp}\mu_{0,h}}\mu_{0,h}(\{x\})\right)}{K_h}\to0.
\]
Thus the initialization serves primarily as an exploration measure: its support determines which parts of the state space can be explored, while
its smallest atom controls the KL cost of redistributing mass.

This is a desirable robustness property. In
particular, if the MFE is unique, then any admissible sequence of
initializations with dense effective support and controlled minimal mass leads
to the same limiting equilibrium. If the MFE is not unique, the theorem only
asserts that every weak subsequential limit is an MFE; it does not rule out
initialization-dependent, mesh-dependent, or subsequence-dependent selection
among multiple equilibria.

\begin{rmk}
Theorem~\ref{thm:mesh-refinement-MFE} may be viewed as a constructive proof of
existence of an MFE. 
Existence of a mean field equilibrium is also guaranteed by a standard
Kakutani--Fan--Glicksberg fixed-point argument
under compactness of \(\Omega\), spatial continuity
\(J(\mu,\cdot)\in C(\Omega)\), and weak-to-uniform continuity of \(J\)
in the measure variable. 
\end{rmk}

Theorem~\ref{thm:mesh-refinement-MFE} is qualitative: it identifies the weak limit points obtained as the mesh is refined and the iteration horizon diverges, but it does not quantify the error at finite \(h\) and \(K\). To obtain such a bound, we must control the error incurred when an arbitrary comparison measure on \(\Omega\) is projected onto the mesh. Assumption~\ref{ass:W1-lip} controls the resulting perturbation in the measure argument of \(J\); the following assumption provides the corresponding control in the spatial argument.

\begin{assumption}[Spatial Lipschitz continuity]
\label{ass:spatial-lipschitz}
Assume that \(J\) is uniformly Lipschitz in the spatial variable:
there exists \(L_x>0\) such that
\[
    |J(\rho,x)-J(\rho,y)|
    \le
    L_x |x-y|,
    \qquad
    \forall \rho\in\mathcal P(\Omega),\quad
    \forall x,y\in\Omega.
\]
\end{assumption}

The finite-state ergodic estimate controls the optimization error, while the measure and spatial Lipschitz assumptions control the projection error. Combining these two contributions yields a uniform Minty residual bound on the full space \(\mathcal P(\Omega)\).

\begin{thm}\label{thm:quantitative-mesh-approximate-MFE} 
Suppose Assumptions~\ref{ass:LL-mono},~\ref{ass:W1-lip}, and~\ref{ass:spatial-lipschitz} hold. Let \[ (X_h,Q_h,\mu_{0,h},m_h)_{h>0}, \qquad 0<\lambda<\frac{1}{L_mD_\Omega}, \] and the KL Mirror-Prox iterates \[ (\mu_{k,h},\nu_{k,h})_{k\ge0} \] be as in Theorem~\ref{thm:mesh-refinement-MFE}. For \(h>0\) and \(K\ge1\), define \[ \bar\nu_{K,h} := \frac1K\sum_{k=0}^{K-1}\nu_{k,h}. \] Then \[ \bigl\langle J(\eta,\cdot), \eta-\bar\nu_{K,h} \bigr\rangle \ge -\varepsilon_{K,h}, \qquad \forall\eta\in\mathcal P(\Omega), \] where \[ \varepsilon_{K,h} := \frac{\log(1/m_h)}{\lambda K} +(2L_m+L_x)h. \]
\end{thm}

Corollary~\ref{cor:finite-mesh-approximate-MFE} and Theorem~\ref{thm:quantitative-mesh-approximate-MFE} provide uniform Minty residual bounds over the entire \(\mathcal P(\Omega)\). To interpret this residual as an approximate mean field equilibrium, we next record the precise relations among approximate MFE, variational inequality, and Minty variational inequality conditions.

\subsection{\texorpdfstring{\(\varepsilon\)-MFE}{epsilon-MFE}, \texorpdfstring{\(\varepsilon\)-VI}{epsilon-VI}, and \texorpdfstring{\(\varepsilon\)-MVI}{epsilon-MVI}}\label{subsec:eps-MFE}

The previous subsection yielded two types of Minty guarantees: exact conditions for weak limit points and uniform residual bounds for finite ergodic averages. We now translate both into equilibrium statements. 
To this end, we introduce approximate versions of the MFE, VI, and MVI conditions and establish the implications among them. These notions and relations are closely related to the approximate VI, MVI, and equilibrium framework studied in \cite{bigi2023approximate}; our results extend this connection to the mean field setting. The resulting implications identify the limiting Minty solutions obtained above as MFEs and convert the finite Minty residual bounds into quantitative approximate-MFE guarantees. Figure~\ref{fig:implication-relations} summarizes the exact and approximate implications.

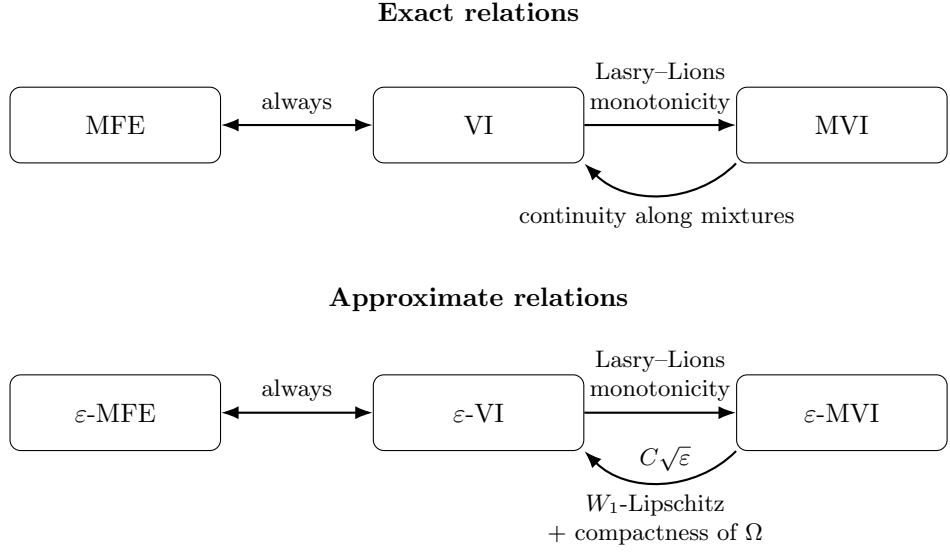
\begin{figure}[t]
\centering
\begin{tikzpicture}[
    >=Latex,
    font=\normalsize,
    every node/.style={align=center},
    box/.style={
        draw,
        rounded corners,
        minimum width=2.5cm,
        minimum height=1.0cm,
        text width=2.5cm,
        inner sep=4pt
    },
    title/.style={
        font=\bfseries
    },
    lab/.style={
        font=\small,
        inner xsep=2pt,
        inner ysep=1pt
    },
    conn/.style={
        thick
    }
]

%---------------- titles ----------------%
\node[title] at (4.8,1.5) {Exact relations};
\node[title] at (4.8,-2.3) {Approximate relations};

%---------------- exact row ----------------%
\node[box] (MFE) at (0,0)    {MFE};
\node[box] (VI)  at (4.8,0)  {VI};
\node[box] (MVI) at (9.6,0)  {MVI};

\draw[conn,<->] (MFE) -- node[lab, above=3pt] {always} (VI);
\draw[conn,->]  (VI)  -- node[lab, above=3pt] {Lasry--Lions\\ monotonicity} (MVI);
\draw[conn,->]  (MVI.south west) to[out=225,in=315]
    node[lab, pos=0.52, below=3pt] {continuity along mixtures}
    (VI.south east);

%---------------- approximate row ----------------%
\node[box] (eMFE) at (0,-3.8)   {\(\varepsilon\)-MFE};
\node[box] (eVI)  at (4.8,-3.8) {\(\varepsilon\)-VI};
\node[box] (eMVI) at (9.6,-3.8) {\(\varepsilon\)-MVI};

\draw[conn,<->] (eMFE) -- node[lab, above=3pt] {always} (eVI);
\draw[conn,->]  (eVI)  -- node[lab, above=3pt] {Lasry--Lions\\ monotonicity} (eMVI);
\draw[conn,->]  (eMVI.south west) to[out=225,in=315]
    node[lab, pos=0.47, above=3pt] {\(C\sqrt{\varepsilon}\)}
    node[lab, pos=0.53, below=3pt] {\(W_1\)-Lipschitz\\
    \(+\) compactness of $\Omega$}
    (eVI.south east);

\end{tikzpicture}
\caption{Relations among MFE, VI, and Minty VI in the exact and approximate settings.}
\label{fig:implication-relations}
\end{figure}

\begin{definition}[\(\varepsilon\)-MFE, \(\varepsilon\)-VI, and \(\varepsilon\)-MVI]
Let \(\varepsilon\ge0\). We say that \(\mu^\varepsilon\in\mathcal P(\Omega)\)
is an \(\varepsilon\)-mean field equilibrium (\(\varepsilon\)-MFE) if
\[
    \int_\Omega J(\mu^\varepsilon,x)\,\mu^\varepsilon(dx)
    \le
    \min_{x\in\Omega}J(\mu^\varepsilon,x)+\varepsilon.
\]
We say that \(\mu^\varepsilon\) satisfies the \(\varepsilon\)-variational
inequality (\(\varepsilon\)-VI) if
\[
    \langle J(\mu^\varepsilon,\cdot),\eta-\mu^\varepsilon\rangle
    \ge
    -\varepsilon,
    \qquad
    \forall \eta\in\mathcal P(\Omega).
\]
We say that \(\mu^\varepsilon\) satisfies the \(\varepsilon\)-Minty
variational inequality (\(\varepsilon\)-MVI) if
\[
    \langle J(\eta,\cdot),\eta-\mu^\varepsilon\rangle
    \ge
    -\varepsilon,
    \qquad
    \forall \eta\in\mathcal P(\Omega).
\]
In particular, when \(\varepsilon=0\), we recover the exact notions: we say that \(\mu\in\mathcal P(\Omega)\) satisfies the variational inequality \emph{(VI)} if \[ \langle J(\mu,\cdot),\eta-\mu\rangle\geq0, \qquad \forall \eta\in\mathcal P(\Omega), \] and the Minty variational inequality \emph{(MVI)} if \[ \langle J(\eta,\cdot),\eta-\mu\rangle\geq0, \qquad \forall \eta\in\mathcal P(\Omega). \]
\end{definition}

\begin{proposition}[\(\varepsilon\)-MFE is equivalent to \(\varepsilon\)-VI]
\label{prop:eps-mfe-equivalent-eps-vi}
Let \(\varepsilon\ge0\) and let \(\mu^\varepsilon\in\mathcal P(\Omega)\).
Then \(\mu^\varepsilon\) is an \(\varepsilon\)-mean field equilibrium if and
only if it satisfies the \(\varepsilon\)-variational inequality
\[
    \langle J(\mu^\varepsilon,\cdot),\eta-\mu^\varepsilon\rangle
    \ge
    -\varepsilon,
    \qquad
    \forall \eta\in\mathcal P(\Omega).
\]
\end{proposition}
In particular, Proposition~\ref{prop:eps-mfe-equivalent-eps-vi} includes the exact implication \[ \mathrm{MFE}\Longleftrightarrow\mathrm{VI} \] as the special case \(\varepsilon=0\).

\begin{proposition}[Exact Minty VI implies exact VI]
\label{prop:mvi-implies-vi}
Let \(\mu\in\mathcal P(\Omega)\) satisfy the exact Minty variational inequality \[ \langle J(\eta,\cdot),\eta-\mu\rangle \geq 0, \qquad \forall \eta\in\mathcal P(\Omega). \] 
Assume that, for every \(\eta\in\mathcal P(\Omega)\), \[ \lim_{t\downarrow0} \left\langle J\bigl((1-t)\mu+t\eta,\cdot\bigr), \eta-\mu \right\rangle = \langle J(\mu,\cdot),\eta-\mu\rangle. \] 
Then \(\mu\) satisfies the exact variational inequality \[ \langle J(\mu,\cdot),\eta-\mu\rangle \geq 0, \qquad \forall \eta\in\mathcal P(\Omega). \]
\end{proposition}

The continuity assumption in Proposition~\ref{prop:mvi-implies-vi} requires only right continuity of the relevant duality pairing along affine mixtures starting at \(\mu\). In particular, it follows from Assumption~\ref{ass:W1-lip}.
Indeed, for \( \eta_t:=(1-t)\mu+t\eta\), we have 
\[ \begin{aligned} 
\left| \left\langle J(\eta_t,\cdot)-J(\mu,\cdot), \eta-\mu \right\rangle \right| 
&\leq \|J(\eta_t,\cdot)-J(\mu,\cdot)\|_\infty \|\eta-\mu\|_1 \\ 
&\leq L_m W_1(\eta_t,\mu) \cdot 2 \\ 
&\leq 2L_m tW_1(\eta,\mu) \longrightarrow 0 \qquad\text{as }t\downarrow0. \end{aligned} \]

Consequently, in the finite-state, full-support setting of Corollary~\ref{cor:cluster-point-mintyVI} and the general compact setting of Theorem~\ref{thm:mesh-refinement-MFE}, every weak limit point satisfies the VI and hence is a mean field equilibrium.

\begin{proposition}[$\varepsilon$-VI implies $\varepsilon$-MVI]
\label{prop:eps-vi-implies-eps-mvi}
Assume Lasry--Lions monotonicity. Let \(\varepsilon\ge0\), and suppose that
\(\mu^\varepsilon\in\mathcal P(\Omega)\) satisfies the \(\varepsilon\)-VI:
\[
    \langle J(\mu^\varepsilon,\cdot),\eta-\mu^\varepsilon\rangle
    \ge
    -\varepsilon,
    \qquad
    \forall \eta\in\mathcal P(\Omega).
\]
Then \(\mu^\varepsilon\) satisfies the \(\varepsilon\)-Minty VI:
\[
    \langle J(\eta,\cdot),\eta-\mu^\varepsilon\rangle
    \ge
    -\varepsilon,
    \qquad
    \forall \eta\in\mathcal P(\Omega).
\]
\end{proposition}
In particular, Proposition~\ref{prop:eps-vi-implies-eps-mvi} includes the exact implication \[ \mathrm{VI}\Longrightarrow\mathrm{MVI} \] as the special case \(\varepsilon=0\).

\begin{proposition}[Approximate MVI implies approximate VI]
\label{prop:eps-mvi-implies-delta-vi}
Suppose Assumption~\ref{ass:W1-lip} holds. Let \(\varepsilon>0\), and let
\(\mu^\varepsilon\in\mathcal P(\Omega)\) satisfy the \(\varepsilon\)-MVI:
\[
    \langle J(\eta,\cdot),\eta-\mu^\varepsilon\rangle
    \ge
    -\varepsilon,
    \qquad
    \forall \eta\in\mathcal P(\Omega).
\]
Assume in addition that
\[
    \varepsilon\le 2L_mD_\Omega.
\]
Then \(\mu^\varepsilon\) satisfies the \(\delta\)-VI:
\[
    \langle J(\mu^\varepsilon,\cdot),\eta-\mu^\varepsilon\rangle
    \ge
    -\delta,
    \qquad
    \forall \eta\in\mathcal P(\Omega),
\]
with
\[
    \delta=2\sqrt{2L_mD_\Omega\,\varepsilon}.
\]
\end{proposition}

We now translate the MVI guarantees in Corollary~\ref{cor:finite-mesh-approximate-MFE} and Theorem~\ref{thm:quantitative-mesh-approximate-MFE} into MFE guarantees. The key conclusion is that, for every prescribed tolerance \(\delta>0\), one can choose the mesh size \(h\) sufficiently small and the iteration count \(K\) sufficiently large so that the proposed KL Mirror-Prox algorithm produces a \(\delta\)-MFE.

For the finite-state problem, Corollary~\ref{cor:finite-mesh-approximate-MFE}
shows that the
ergodic average \(\bar{\nu}_K\) satisfies an \(\varepsilon_K\)-MVI for the finite-state problem, where
\[
\varepsilon_K
=
\frac{\log(1/m)}{\lambda K}.
\]
Whenever $\varepsilon_K\le 2L_mD_\Omega$,
Propositions~\ref{prop:eps-mvi-implies-delta-vi}
and~\ref{prop:eps-mfe-equivalent-eps-vi} imply that $\bar{\nu}_K$ satisfies a
$\delta_K$-VI and hence is a
$\delta_K$-MFE, with
\[
\delta_K
=
2\sqrt{2L_mD_\Omega\,\varepsilon_K}
=
2\sqrt{
\frac{2L_mD_\Omega\log(1/m)}
{\lambda K}
}.
\]
Since $m$ is fixed, this yields
\[
\delta_K=O(K^{-1/2}),
\]
and hence $\delta_K\to 0$ as $K\to\infty$. 
In particular, for any $\delta>0$, let
\[
\varepsilon_\delta^\star
:=
\min\left\{
2L_mD_\Omega,\,
\frac{\delta^2}{8L_mD_\Omega}
\right\}.
\]
Then any
\[
K
\ge
\frac{\log(1/m)}
{\lambda\varepsilon_\delta^\star}
\]
ensures that $\bar{\nu}_K$ is a $\delta$-MFE. Thus, on a fixed finite
state space, the algorithm computes arbitrarily accurate approximate MFEs.

For the general compact-state problem,
Theorem~\ref{thm:quantitative-mesh-approximate-MFE} shows that the
ergodic average $\bar{\nu}_{K,h}$ satisfies an
$\varepsilon_{K,h}$-MVI for the general compact state space problem, where
\[
\varepsilon_{K,h}
=
\frac{\log(1/m_h)}{\lambda K}
+
(2L_m+L_x)h.
\]
Whenever $\varepsilon_{K,h}\le 2L_mD_\Omega$, Propositions~\ref{prop:eps-mvi-implies-delta-vi}
and~\ref{prop:eps-mfe-equivalent-eps-vi} imply
that $\bar{\nu}_{K,h}$ satisfies a
$\delta_{K,h}$-VI and hence is a $\delta_{K,h}$-MFE, with
\[
\delta_{K,h}
=
2\sqrt{2L_mD_\Omega\,\varepsilon_{K,h}}
=
O\left(
\sqrt{
h+\frac{\log(1/m_h)}{K}
}
\right).
\]
Given any $\delta>0$, let
\[
\varepsilon_\delta^\star
:=
\min\left\{
2L_mD_\Omega,\,
\frac{\delta^2}{8L_mD_\Omega}
\right\}.
\]
It is sufficient to choose
\[
h
\le
\frac{\varepsilon_\delta^\star}
{2(2L_m+L_x)}
\]
and, after fixing the mesh and initialization (and hence \(m_h\)), choose
\[
K
\ge
\frac{2\log(1/m_h)}
{\lambda\varepsilon_\delta^\star}.
\]
Then $\varepsilon_{K,h}\le\varepsilon_\delta^\star$ and consequently
$\delta_{K,h}\le\delta$. Therefore, for every $\delta>0$, a sufficiently
fine mesh followed by sufficiently many KL Mirror-Prox iterations produces
a $\delta$-MFE of the original compact-state problem. In this sense, the
proposed algorithm computes MFEs to arbitrary accuracy.

\section{Metric Convergence under Strong Lasry--Lions Monotonicity}\label{sec:strong-LL-mono}

The preceding section provides convergence results in terms of a global Minty residual and an approximate equilibrium certificate. These estimates measure how closely the ergodic averages satisfy the variational inequality, but they do not directly control the distance between the iterates and an equilibrium.
We now impose a strong form of Lasry--Lions monotonicity, under which such metric estimates become available.

\begin{assumption}[Strong Lasry--Lions monotonicity]\label{ass:strong-LL-mono}
There exists
\(\alpha>0\) such that, for all \(\rho,\sigma\in\mathcal P(\Omega)\),
\[
    \langle J(\rho,\cdot)-J(\sigma,\cdot),\rho-\sigma\rangle
    \ge
    \alpha W_1(\rho,\sigma)^2.
\]
\end{assumption}
The strong monotonicity condition has two immediate consequences. First, it guarantees uniqueness of the mean field equilibrium. Second, the monotonicity pairing directly controls a metric error, allowing the one-step KL estimate to yield convergence rates for both individual and ergodic iterates rather than only qualitative convergence of ergodic averages.

Under this stronger assumption, we first study the algorithm on a fixed finite mesh and establish eventual sublinear rates of last-iterate convergence to the unique finite-mesh VI solution and, when that solution has full support, an eventual geometric rate. We then impose spatial Lipschitz continuity of the cost and derive an estimate comparing the finite-mesh VI solution with the solution of the original problem. Finally, we combine the optimization and discretization estimates to obtain explicit finite-iteration bounds for approximating the mean field equilibrium.

First, we consider the problem on a fixed finite mesh. Applying Lemma~\ref{lem:one-step-Mirror-Prox} with the finite-mesh VI solution as the comparison measure and using strong monotonicity yields a Fej\'er-type descent inequality, from which convergence of the full sequence to the unique solution follows. The resulting last-iterate rates then rely on finite-dimensional equivalences between KL divergence, total variation, and \(W_1\).

\begin{thm}
\label{thm:finite-mesh-last-iterate}
Suppose Assumptions~\ref{ass:W1-lip} and~\ref{ass:strong-LL-mono} hold.
Let
\(
    X_h=\{x_1,\dots,x_{N_h}\}\subset\Omega
\)
with $N_h\ge2$,
and let
\[
    \delta_h:=\min_{i\neq j}|x_i-x_j|.
\]
Let
\[
    0<\lambda<\frac{1}{L_mD_\Omega}.
\]
Then the finite-mesh VI admits a unique solution \(\mu_h^\ast\in\mathcal P(X_h)\),
characterized by
\[
    \langle J(\mu_h^\ast,\cdot),\eta_h-\mu_h^\ast\rangle\ge0,
    \qquad
    \forall \eta_h\in\mathcal P(X_h).
\]
Let
\((\mu_{k,h},\nu_{k,h})_{k\ge0}\) be the finite-mesh KL Mirror-Prox iterates initialized from \(\mu_{0,h}\), where \(\mu_{0,h}\) has full support on \(X_h\).
Then
\[
    W_1(\mu_{k,h},\mu_h^\ast)\to0,
    \qquad
    W_1(\nu_{k,h},\mu_h^\ast)\to0.
\]

Define
\[
    D_{k,h}:=D_{\rm KL}(\mu_h^\ast\|\mu_{k,h}),
    \qquad
    m_h^\ast:=\min_{x\in\supp\mu_h^\ast}\mu_h^\ast(x),
\]
and
\[
    \beta:=
    \frac12
    \min\left\{
        \lambda\alpha,
        \frac{1-\lambda L_mD_\Omega}{2D_\Omega^2}
    \right\},\qquad
    \gamma_h
    :=
    \frac14
    \min\left\{
        \frac{\lambda\alpha\delta_h^2}{2},
        1-\lambda L_mD_\Omega
    \right\}.
\]
Then the following estimates hold.

\begin{enumerate}
\item
For every \(k\ge0\),
\begin{equation}\label{eq:finite-mesh-W1-descent}
    \beta
    W_1(\mu_{k,h},\mu_h^\ast)^2
    \le
    D_{k,h}-D_{k+1,h}.
\end{equation}

\item
There exists \(k_{0,h}\in\mathbb N\) satisfying
\[
    k_{0,h}
    \le
    \left\lceil
        \frac{256D_{0,h}}{\gamma_h (m_h^\ast)^4}
    \right\rceil
\]
such that, for every \(k\ge k_{0,h}\),
\[
    D_{k,h}
    \le
    \frac{1}{
        \frac{8}{(m_h^\ast)^2}
        +
        \frac{\gamma_h}{4}(k-k_{0,h})
    },
\]
and
\[
    W_1(\mu_{k,h},\mu_h^\ast)
    \le
    D_\Omega
    \sqrt{
        \frac{2}{
            \frac{8}{(m_h^\ast)^2}
            +
            \frac{\gamma_h}{4}(k-k_{0,h})
        }
    }.
\]

\item
If, in addition, \(\mu_h^\ast\) has full support on \(X_h\), then there exists
\(k_{0,h}\in\mathbb N\) satisfying
\[
    k_{0,h}
    \le
    \left\lceil
        \frac{8D_{0,h}}{\gamma_h (m_h^\ast)^3}
    \right\rceil
\]
such that, for every \(k\ge k_{0,h}\),
\[
    D_{k,h}
    \le
    \frac{(m_h^\ast)^2}{8}
    \left(1-\gamma_h m_h^\ast\right)^{k-k_{0,h}},
\]
and
\[
    W_1(\mu_{k,h},\mu_h^\ast)
    \le
    D_\Omega
    \frac{m_h^\ast}{2}
    \left(1-\gamma_h m_h^\ast\right)^{(k-k_{0,h})/2}.
\]
\end{enumerate}
Finally,
\[
D_{0,h}
\le
-\log\left(
\min_{x\in\supp\mu_h^\ast}\mu_{0,h}(x)
\right)
\le
-\log\left(
\min_{x\in X_h}\mu_{0,h}(x)
\right).
\]
\end{thm}

\begin{rmk}
The full-support assumption on \(\mu_h^\ast\) can be replaced by other conditions that yield eventual geometric convergence. One such alternative is a suitable active-set gap, which ensures that mass outside the support of \(\mu_h^\ast\) is suppressed at a geometric rate.
Concretely, suppose that there exist a neighborhood \(\mathcal U_h\) of \(\mu_h^\ast\)
and a constant \(\Delta_h>0\) such that, for every \(\rho\in\mathcal U_h\),
\[
    J(\rho,x)
    \ge
    J(\rho,y)+\Delta_h,
    \qquad
    \forall x\notin \supp(\mu_h^\ast), \quad \forall y\in \supp(\mu_h^\ast).
\]
Then the exponential tilt suppresses the inactive mass by a fixed factor
once the iterates enter \(\mathcal U_h\). Such a condition can restore
geometric convergence even when \(\mu_h^\ast\) does not have full support on \(X_h\).

Without such a gap, geometric convergence cannot be expected. Consider $X=\{0,1\}$. For
\[
    \mu=(1-p)\delta_0+p\delta_1,
    \qquad p\in[0,1],
\]
define
\[
    J(\mu,0)=0,
    \qquad
    J(\mu,1)=p.
\]
Then the unique equilibrium is
\[
    \mu^\ast=\delta_0.
\]
The operator is strongly monotone in \(L^1\), since for
\[
    \mu=(1-p)\delta_0+p\delta_1,
    \qquad
    \nu=(1-q)\delta_0+q\delta_1,
\]
one has
\[
    \langle J(\mu,\cdot)-J(\nu,\cdot),\mu-\nu\rangle
    =
    (p-q)^2
    =
    \frac14\|\mu-\nu\|_1^2.
\]
However, there is no active-set gap at equilibrium, because
\[
    J(\mu^\ast,1)-J(\mu^\ast,0)=0.
\]
For the KL Mirror-Prox iteration, if
\[
    \mu_k=(1-p_k)\delta_0+p_k\delta_1,
\]
then
\[
    q_k
    :=
    \nu_k(\{1\})
    =
    \frac{p_k e^{-\lambda p_k}}
    {1-p_k+p_k e^{-\lambda p_k}},
\]
and
\[
    p_{k+1}
    =
    \frac{p_k e^{-\lambda q_k}}
    {1-p_k+p_k e^{-\lambda q_k}}.
\]
As \(p_k\downarrow0\),
\[
    q_k=p_k-\lambda p_k^2+O(p_k^3),
\]
and therefore
\[
    p_{k+1}
    =
    p_k-\lambda p_k^2+O(p_k^3).
\]
Thus
\[
    p_k\sim \frac1{\lambda k}.
\]
Consequently,
\[
    D_{\rm KL}(\mu^\ast\|\mu_k)
    =
    -\log(1-p_k)
    \sim
    p_k
    \sim
    \frac1{\lambda k}.
\]
This example satisfies strong monotonicity but has no active-set gap, and it
exhibits \(O(1/k)\) KL decay rather than geometric convergence.
\end{rmk}

Under the additional spatial Lipschitz condition in Assumption~\ref{ass:spatial-lipschitz}, Proposition~\ref{prop:direct-mesh-stability} compares the solutions of the continuous and discrete variational inequalities via the projection \((Q_h)_\#\mu^\ast\). This estimate is independent of the algorithm and separates discretization error from optimization error.

\begin{proposition}
\label{prop:direct-mesh-stability}
Suppose Assumptions~\ref{ass:strong-LL-mono} and~\ref{ass:spatial-lipschitz} hold.
Let $X_h$ be a finite subset of $\Omega$ and $Q_h:\Omega\to X_h$ be a measurable mesh projection
satisfying
\[
    \sup_{x\in\Omega}
    \lvert Q_h(x)-x\rvert
    \leq h.
\]
Let \(\mu^\ast\in\mathcal P(\Omega)\)
solve the continuous VI
\[
    \langle J(\mu^\ast,\cdot),\eta-\mu^\ast\rangle\ge0,
    \qquad
    \forall \eta\in\mathcal P(\Omega),
\]
and let \(\mu_h^\ast\in\mathcal P(X_h)\) solve the finite-mesh VI
\[
    \langle J(\mu_h^\ast,\cdot),\eta_h-\mu_h^\ast\rangle\ge0,
    \qquad
    \forall \eta_h\in\mathcal P(X_h).
\]
Then
\[
    W_1(\mu_h^\ast,\mu^\ast)
    \le
    \sqrt{\frac{L_x}{\alpha}h}.
\]
\end{proposition}

Combining Theorem~\ref{thm:finite-mesh-last-iterate} and Proposition~\ref{prop:direct-mesh-stability} yields the last-iterate bounds. Best-iterate bounds follow from summing the Fejér inequality, while ergodic bounds use convexity of \(W_1\).

\begin{thm}
\label{thm:finite-iteration-continuous-error}
Assume the hypotheses of Theorem~\ref{thm:finite-mesh-last-iterate} and Proposition~\ref{prop:direct-mesh-stability}, and let
\(\mu^\ast\in\mathcal P(\Omega)\) be the continuous equilibrium. For
\(K\ge1\), define
\[
    \bar\mu_{K,h}:=\frac1K\sum_{k=0}^{K-1}\mu_{k,h}.
\]
Then the following estimates hold.

\begin{enumerate}
\item
Let \(k_{0,h}\) be the index from item \(2\) of
Theorem~\ref{thm:finite-mesh-last-iterate}. Then, for every
\(k\ge k_{0,h}\),
\[
    W_1(\mu_{k,h},\mu^\ast)
    \le
    D_\Omega
    \sqrt{
        \frac{2}{
            \frac{8}{(m_h^\ast)^2}
            +
            \frac{\gamma_h}{4}(k-k_{0,h})
        }
    }
    +
    \sqrt{\frac{L_x}{\alpha}h}.
\]
If, in addition, \(\mu_h^\ast\) has full support on \(X_h\), and if
\(k_{0,h}\) is the index from item \(3\) of
Theorem~\ref{thm:finite-mesh-last-iterate}, then, for every
\(k\ge k_{0,h}\),
\[
    W_1(\mu_{k,h},\mu^\ast)
    \le
    D_\Omega
    \frac{m_h^\ast}{2}
    \left(1-\gamma_h m_h^\ast\right)^{(k-k_{0,h})/2}
    +
    \sqrt{\frac{L_x}{\alpha}h}.
\]

\item
For every \(K\ge1\),
\[
    \min_{0\le k\le K-1}
    W_1(\mu_{k,h},\mu^\ast)
    \le
    \sqrt{
        \frac{D_{0,h}}{
            \beta K
        }
    }
    +
    \sqrt{\frac{L_x}{\alpha}h}.
\]
Consequently, if
\[
    h\le \frac{\alpha\varepsilon^2}{4L_x}
\]
and
\[
    K\ge
    \frac{4D_{0,h}}{
        \varepsilon^2 \beta
    },
\]
then there exists \(0\le k\le K-1\) such that
\[
    W_1(\mu_{k,h},\mu^\ast)\le\varepsilon.
\]

\item
For every \(K\ge1\),
\[
    W_1(\bar\mu_{K,h},\mu^\ast)
    \le
    \sqrt{
        \frac{D_{0,h}}{
            \beta K
        }
    }
    +
    \sqrt{\frac{L_x}{\alpha}h}.
\]
Consequently, if
\[
    h\le \frac{\alpha\varepsilon^2}{4L_x}
\]
and
\[
    K\ge
    \frac{4D_{0,h}}{
        \varepsilon^2 \beta
    },
\]
then
\[
    W_1(\bar\mu_{K,h},\mu^\ast)\le\varepsilon.
\]
\end{enumerate}

\end{thm}

Theorem~\ref{thm:finite-iteration-continuous-error} gives three types of
bounds: a last-iterate bound, a best-iterate bound, and an averaged-iterate
bound. Each estimate decomposes the total error into an algorithmic error and
the mesh-approximation error \(W_1(\mu_h^\ast,\mu^\ast)\).

The last-iterate estimate is a genuine a priori estimate, since the entry
index \(k_{0,h}\) is bounded explicitly in Theorem~\ref{thm:finite-mesh-last-iterate}. However, this entry bound can
deteriorate substantially as the mesh is refined. 
For a quasi-uniform mesh in dimension \(d\), one typically has
\[
    \delta_h\asymp h,
    \qquad
    \gamma_h\asymp \delta_h^2\asymp h^2,
    \qquad
    N_h:=|X_h|\asymp h^{-d}.
\]
If, in addition,
\[
    m_h^\ast\asymp h^d,
    \qquad
    D_{0,h}=O(d|\log h|),
\]
as may occur for suitably regular full-support equilibria and a uniform
initialization, then the entry bounds scale as
\[
    k_{0,h}
    \lesssim 
    \frac{D_{0,h}}{\gamma_h(m_h^\ast)^4}
    \asymp
    h^{-(4d+2)}d|\log h|
\]
in the general-support case, and
\[
    k_{0,h}
    \lesssim
    \frac{D_{0,h}}{\gamma_h(m_h^\ast)^3}
    \asymp
    h^{-(3d+2)}d|\log h|
\]
in the full-support case. Thus explicit mesh-dependent complexity for the last-iterate estimate can be quite
conservative.

By contrast, the best-iterate and averaged-iterate estimates are cleaner a priori bounds. They use the direct
\(W_1\)-descent estimate \eqref{eq:finite-mesh-W1-descent} and do not use the
local reverse-KL region and therefore do not contain the mesh-separation factor
\(\delta_h\). Thus, they provide a more useful practical guide for choosing the mesh size \(h\) and the number of iterations \(K\) before running
the algorithm.

Indeed, these estimates suggest balancing the algorithmic error
\[
    \sqrt{\frac{D_{0,h}}{\beta K}}
\]
against the mesh approximation error
\[
    \sqrt{\frac{L_x}{\alpha}h}.
\]
For example, suppose that the initialization is uniform,
\(D_\Omega=O(\sqrt d)\), the problem constants are dimension-independent,
and \(\lambda\) is chosen as a fixed fraction of
\(1/(L_mD_\Omega)\). Then
\[
    D_{0,h}=O(d|\log h|),
    \qquad
    \beta^{-1}=O(d),
\]
choosing
\[
    h\asymp \varepsilon^2,
    \qquad
    K\asymp d^2\varepsilon^{-2}|\log\varepsilon|
\]
gives an \(O(\varepsilon)\) guarantee for the averaged iterate and guarantees
that at least one of the first \(K\) iterates is \(O(\varepsilon)\)-accurate.

Thus, under the scaling assumptions above, the iteration count for the
best-iterate and averaged-iterate guarantees depends on the dimension only
through multiplicative polynomial factors; the dimension does not enter the
exponent of \(\varepsilon^{-1}\). This indicates a mild dimension dependence
at the level of iteration complexity. However, note that on a quasi-uniform grid, \(N_h\asymp h^{-d}\), so the storage and per-iteration
cost may still scale like a negative power of \(\varepsilon\) whose exponent
depends on \(d\).

\section{Equilibrium Selection with Relative Entropy}\label{sec:MFE-selection}
In this section, we consider equilibrium selection on a general compact state space and study how a KL-type Tikhonov penalty can select a distinguished mean field equilibrium. More precisely, we choose a reference measure \(\rho\in\mathcal P(\Omega)\) with full support and consider equilibria \(\mu\) satisfying \(D_{\mathrm{KL}}(\mu\|\rho)<\infty\). When this minimum is attained uniquely, the vanishing-penalty limit selects the equilibrium that minimizes relative entropy with respect to \(\rho\) over the MFE solution set.

Let \(\rho\in\mathcal P(\Omega)\) be a reference measure with full support.
For each finite mesh \(X_h\subseteq\Omega\) and measurable mesh projection $Q_h:\Omega\to X_h$ such that
\[
    \sup_{x\in\Omega} |Q_h(x)-x|\le h,
\]
define
\[
    \rho_h:=Q_{h\#}\rho\in\mathcal P(X_h).
\]
We assume, for simplicity, that \(\rho_h(x)>0\) for every \(x\in X_h\). This entails no loss of generality. Indeed, if some mesh point has zero \(\rho_h\)-mass, one may replace \(X_h\) by \(\operatorname{supp}(\rho_h)\) and modify \(Q_h\) on a \(\rho\)-null set. The resulting measurable map still satisfies
\[
\sup_{x\in\Omega} |Q_h(x)-x|\le h
\]
and has the same pushforward measure \(\rho_h\).

Define the full-space and discrete entropy functionals by
\[
    R(\mu):=D_{\rm KL}(\mu\|\rho),
    \qquad
    R_h(\mu_h):=D_{\rm KL}(\mu_h\|\rho_h).
\]

For each fixed \(h\) and \(\varepsilon>0\), the regularized VI below admits a
unique solution in the relative interior of \(\mathcal P(X_h)\). Existence
follows by applying Brouwer's fixed-point theorem to the associated continuous
Gibbs map. Uniqueness follows from Lasry--Lions monotonicity of \(J\) and
strict monotonicity of the first variation of relative entropy.

\begin{thm}
\label{thm:joint-mesh-tikhonov-selection}
Suppose Assumptions~\ref{ass:LL-mono},~\ref{ass:W1-lip}, and~\ref{ass:spatial-lipschitz} hold.
Let
\[
    \mathcal S
    :=
    \left\{
    \mu\in\mathcal P(\Omega):
    \langle J(\mu,\cdot),\eta-\mu\rangle\ge0
    \quad \forall \eta\in\mathcal P(\Omega)
    \right\}
\]
be the equilibrium set. Assume that
\[
    \mu^\dagger
    =
    \operatorname*{argmin}_{\mu\in\mathcal S}
    D_{\rm KL}(\mu\|\rho)
\]
exists, is unique, and satisfies
\[
    D_{\rm KL}(\mu^\dagger\|\rho)<\infty.
\]

For each \(h\) and \(\varepsilon_h>0\), let \(\mu_h^{\varepsilon_h}\in\mathcal P(X_h)\)
solve the regularized VI:
\[
    \left\langle
        J(\mu_h^{\varepsilon_h},\cdot)
        +
        \varepsilon_h
        \log\frac{\mu_h^{\varepsilon_h}}{\rho_h},
        \eta_h-\mu_h^{\varepsilon_h}
    \right\rangle
    \ge0,
    \qquad
    \forall \eta_h\in\mathcal P(X_h).
\]
Assume the balance conditions
\[
    \frac{h}{\varepsilon_h}\to0,\qquad
    \varepsilon_h \log\frac1{\rho_{\min,h}}\to0,
\]
as $h\to 0$, where \(\rho_{\min,h}:=\min_{x\in X_h}\rho_h(x)\).
Then
\[
    \mu_h^{\varepsilon_h}\Rightarrow \mu^\dagger.
\]
\end{thm}

\begin{rmk}
For a quasi-uniform \(d\)-dimensional mesh and a reference density
bounded above and below, one typically has
\[
    \rho_{\min,h}\asymp h^d,
\]
so $\varepsilon_h=\sqrt h$ is one admissible choice satisfying
\[
    \frac{h}{\varepsilon_h}\to0,\qquad
    \varepsilon_h\log\frac1{\rho_{\min,h}}\to0.
\]
\end{rmk}

%\begin{rmk}
%\label{rmk:approximate-regularized-vi}
The exact solution assumption in Theorem~\ref{thm:joint-mesh-tikhonov-selection}
can be weakened. It is
sufficient that, for each \(h\),
\(\mu_h^{\varepsilon_h}\in\mathcal P(X_h)\) satisfy the discrete regularized
VI up to an error \(r_h\geq 0\), namely
\[
    \left\langle
        J(\mu_h^{\varepsilon_h},\cdot)
        +
        \varepsilon_h
        \log\frac{\mu_h^{\varepsilon_h}}{\rho_h},
        \eta_h-\mu_h^{\varepsilon_h}
    \right\rangle
    \ge
    -r_h,
    \qquad
    \forall \eta_h\in\mathcal P(X_h).
\]
If
\[
    \frac{r_h}{\varepsilon_h}\to0,
\]
then the conclusion of
Theorem~\ref{thm:joint-mesh-tikhonov-selection} remains unchanged:
\[
    \mu_h^{\varepsilon_h}\Rightarrow\mu^\dagger.
\]

Such a sequence can be constructed by running the proposed KL
Mirror-Prox algorithm on the regularized cost
\[
    F_h^{\varepsilon_h}(\mu,\cdot)
    :=
    J(\mu,\cdot)
    +
    \varepsilon_h\log\frac{\mu}{\rho_h}.
\]

Indeed, suppose that \(J\) is bounded on
\(\mathcal P(X_h)\times X_h\), that \(\rho_h\) and the initial iterate
\(\mu_{0,h}\) have full support, and that
\(0<\lambda_h\varepsilon_h<1\). For each fixed \(h\), boundedness of \(J\) and this choice of \(\lambda_h\)
imply a uniform bound on
the oscillation of \(\log(\mu_{k,h}/\rho_h)\) and \(\log(\nu_{k,h}/\rho_h)\) along the trajectories. Consequently, there is a constant \(a_h>0\) such that every coordinate of every predictor and corrector is at least \(a_h\). Since \(X_h\)
is finite, \(F_h^{\varepsilon_h}\) is
\(W_1\)-to-\(L^\infty\) Lipschitz, with some constant \(\widetilde L_h\), on
the compact set of probability vectors whose coordinates are at least
\(a_h/2\). Choose \(\lambda_h\) to satisfy the corresponding trajectory-local
step-size condition.

Repeating the proof of Lemma~\ref{lem:ergodic-minty} for interior comparison
measures gives a restricted Minty residual \(s_{h,K}=O(K^{-1})\) for the
ergodic average \(\bar\nu_{K,h}\). For arbitrary
\(\eta_h\in\mathcal P(X_h)\) and \(t\le1/2\), the mixture
\[
\rho_{h,t}:=(1-t)\bar\nu_{K,h}+t\eta_h
\]
has every coordinate at least \(a_h/2\). Thus the proof of
Proposition~\ref{prop:eps-mvi-implies-delta-vi} can be repeated along this
short segment. For all sufficiently large \(K\), take
\[
t_{h,K}:=
\left(\frac{s_{h,K}}{2\widetilde L_hD_\Omega}\right)^{1/2}
\le\frac12.
\]
The same calculation then yields a regularized VI residual
\[
r_{h,K}
\le
2\sqrt{2\widetilde L_hD_\Omega s_{h,K}}
\longrightarrow0.
\]
Hence, for any sequence
\(\tau_h\downarrow0\), one may choose \(K_h\) sufficiently large so that
\(
r_{h,K_h}\leq \varepsilon_h\tau_h,
\)
and define
\(
\mu_h^{\varepsilon_h}:=\bar\nu_{K_h,h}.
\)
Then \(\mu_h^{\varepsilon_h}\) satisfies the required approximate discrete
regularized VI with \(r_h:=r_{h,K_h}\) and
\(
\frac{r_h}{\varepsilon_h}\le\tau_h\to0.
\)
%\end{rmk}

\begin{rmk}[Hypomonotone case]
The preceding theorem assumes monotonicity. If \(J\) is only
\(\alpha_h\)-hypomonotone on $\mathcal{P}(X_h)$, the argument can
still work if
\[
    \frac{\alpha_h}{\varepsilon_h}\to0.
\]
For a fixed, genuinely hypomonotone problem with \(\alpha>0\), this condition
is incompatible with \(\varepsilon_h\downarrow0\). In that case, entropy may
stabilize the problem for \(\varepsilon_h>\alpha\), but the monotonicity-based
zero-penalty selection argument does not extend to \(\varepsilon_h\downarrow0\).
\end{rmk}

\section{Numerical Experiments}\label{sec:experiment}
In the experiments below, we use the finite-dimensional implementation of the KL Mirror-Prox method developed in Section~\ref{sec:algo}.\footnote{All code is available at \url{https://github.com/Lethargy/kl-mirror-prox-mfe}} For a mesh resolution \(h>0\), let \(X_h=\{x_1,\ldots,x_M\}\subset\Omega\) be a finite mesh satisfying \eqref{eq:mesh-projection}, and identify each measure on \(X_h\) with its vector of weights in \(\Delta_M\). Algorithm~\ref{alg:kl-mirror-extragradient} is the explicit finite-dimensional iteration \eqref{eq:finite-KL-MP}. The prediction step evaluates the cost at the current iterate \(\mu_k\), while the correction step evaluates it at the predictor \(\nu_k\) and reweights the original vector \(p_k\). The normalization steps keep both vectors in \(\Delta_M\), and an initialization with full support retains full support throughout the iteration.

Fix \(h>0\) and write \(X_h=\{x_1,\ldots,x_M\}\). Throughout the algorithm, we suppress the dependence of the iterates on \(h\). Thus \(p_{k,i}\) and \(q_{k,i}\) denote the weights assigned to \(x_i\) at iteration \(k\), consistently with the notation of Section~\ref{sec:algo}.

\begin{algorithm}[H]
\caption{KL Mirror-Prox method}
\label{alg:kl-mirror-extragradient}
\begin{algorithmic}[1]
\State Choose \(p_0\in\Delta_M\) with \(p_{0,i}>0\) for every \(i\), and choose a stepsize \(\lambda\) according to \eqref{eq:step-size-condition}.
\For{\(k=0,1,2,\ldots\)}
    \State $\mu_k \gets \sum_{i=1}^M p_{k,i}\delta_{x_i}$
    \For{\(i=1,\ldots,M\)}
        \State $q_{k,i} \gets p_{k,i}\exp\!\left(-\lambda J(\mu_k,x_i)\right)$
        \Comment{Prediction step}
    \EndFor
    \State $q_k\gets q_k\big/\sum_{j=1}^M q_{k,j}$
    \Comment{Normalize}
    \State $\nu_k \gets \sum_{i=1}^M q_{k,i}\delta_{x_i}$
    \For{\(i=1,\ldots,M\)}
        \State $p_{k+1,i}\gets p_{k,i}\exp\!\left(-\lambda J(\nu_k,x_i)\right)$
        \Comment{Correction step}
    \EndFor
    \State $p_{k+1}\gets p_{k+1}\big/\sum_{j=1}^M p_{k+1,j}$
    \Comment{Normalize}
\EndFor
\end{algorithmic}
\end{algorithm}

\subsection{Arctangent Example}
Consider
\begin{equation}
\label{ex:arctan}
J(\mu,x) = x^2+\int_{\Omega}\arctan(x-y)\mu(dy), \qquad \Omega=[-1,1].
\end{equation}
We verify that the assumptions are satisfied.

The set \(\Omega\) is compact. Moreover, for each \(\mu\in\cP(\Omega)\), continuity of \(x\mapsto J(\mu,x)\) follows from dominated convergence. Thus Assumption~\ref{ass:standing} holds.

Since $\arctan(y-x)=-\arctan(x-y)$,
\begin{equation}
\label{eq:atan-ll-monotone}
\left\langle J(\mu,\cdot)-J(\nu,\cdot),\mu-\nu\right\rangle = \iint_{\Omega^2}\arctan(x-y)(\mu - \nu)(dx)(\mu - \nu)(dy)=0.
\end{equation}
Thus Lasry--Lions monotonicity (Assumption~\ref{ass:LL-mono}) is satisfied.

Since \(y\mapsto\arctan(x-y)\) is \(1\)-Lipschitz, the Kantorovich--Rubinstein duality gives
\[
\left|J(\mu,x)-J(\nu,x)\right| =
\left|\int_\Omega \arctan(x-y)(\mu-\nu)(dy)\right| \leq W_1(\mu,\nu).
\]
Taking the supremum over \(x\in\Omega\) shows that
Assumption~\ref{ass:W1-lip} holds with \(L_m=1\).

Finally, for \(x,z\in\Omega\),
\[
|J(\mu,x)-J(\mu,z)| \leq |x^2-z^2|
+\int_\Omega|\arctan(x-y)-\arctan(z-y)|\mu(dy) \leq 3|x-z|.
\]
Hence Assumption~\ref{ass:spatial-lipschitz} holds with \(L_x=3\). Since \(D_\Omega=2\), the step-size condition becomes \(0<\lambda<\frac12\).

A direct calculation shows that the equilibrium is \(\mu^\ast=\delta_{-1/2}\). To implement Algorithm~\ref{alg:kl-mirror-extragradient}, we take \(M\) equally spaced points \(x_1,\ldots,x_M\) in \([-1,1]\), initialize uniformly, and use the stepsize \(\lambda=0.49\). Since \(m_h=1/M\), \(L_m=1\), \(L_x=3\), and \(h=1/(M-1)\), Theorem~\ref{thm:quantitative-mesh-approximate-MFE} gives
\[
\sup_{\eta \in \cP(\Omega)}\bigl\langle J(\eta,\cdot), \bar\nu_{K,h} - \eta \bigr\rangle \leq \frac{\log M}{\lambda K} + \frac{5}{M-1}.
\]
Moreover, by \eqref{eq:atan-ll-monotone}, $\langle J(\eta,\cdot), \bar\nu_{K,h} - \eta \rangle = \langle J(\bar\nu_{K,h},\cdot), \bar\nu_{K,h} - \eta \rangle$. Thus the Minty gap and the VI gap coincide for this particular choice of $J$. Consequently,
\[
\sup_{\eta \in \cP(\Omega)}\bigl\langle J(\bar\nu_{K,h},\cdot), \bar\nu_{K,h} - \eta \bigr\rangle \leq \frac{\log M}{\lambda K} + \frac{5}{M-1}.
\]
Figure~\ref{fig:arctan-gaps} compares the Minty and VI gaps with the common theoretical bound on an \(M=1001\) mesh, while Figure~\ref{fig:arctan-averaged-measure} shows the evolution of the averaged predictor measure on an \(M=8\) mesh.
\begin{figure}
\centering
\includegraphics[width=\textwidth]{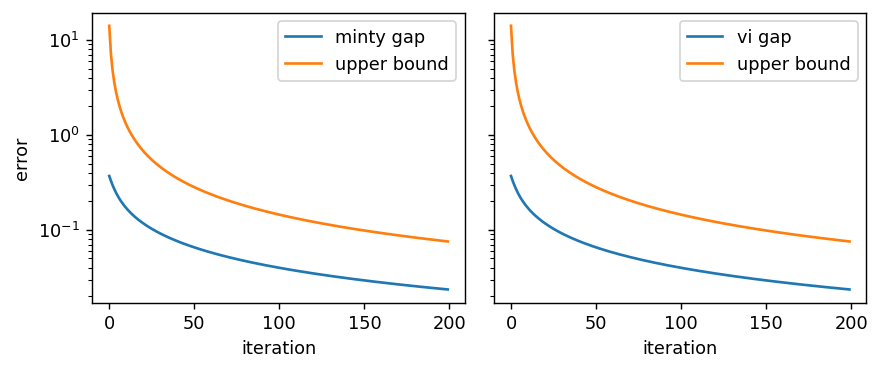}
\caption{Minty gap (left), variational inequality gap (right), and the common upper bound from Theorem~\ref{thm:quantitative-mesh-approximate-MFE} for Example~\eqref{ex:arctan}, using \(M=1001\).}
\label{fig:arctan-gaps}
\end{figure}

\begin{figure}
\centering
\includegraphics[width=\textwidth]{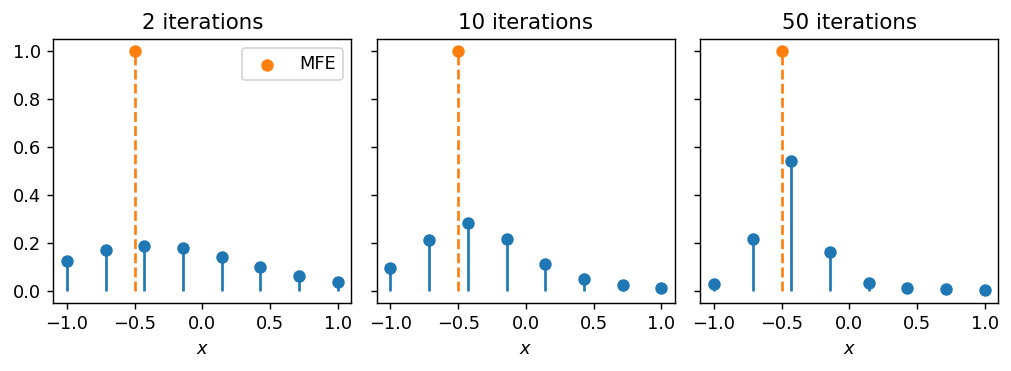}
\caption{Averaged predictor measure $\bar \nu_{K,h}$ for Example~\eqref{ex:arctan} at selected iterations with $M = 8$. The dashed line marks the equilibrium $x^\ast=-1/2$.}
\label{fig:arctan-averaged-measure}
\end{figure}

\subsection{Double-Well Example}
Consider $\Omega = [-2,2]$ and
\begin{equation}
\label{eq:double-well}
J(\mu,x)
=
(x^2-1)^2+\alpha \bar\mu x,
\qquad
\bar\mu:=\int_{\Omega} y\,\mu(dy),
\qquad
\alpha>0.
\end{equation}
The set \(\Omega\) is compact. For each $\mu \in \cP(\Omega)$, $x \mapsto J(\mu,x)$ is polynomial and thus continuous. Assumption~\ref{ass:standing} is satisfied.

For $\mu, \nu \in \cP(\Omega)$, 
\begin{equation*}
\left\langle J(\mu,\cdot)-J(\nu,\cdot),\mu-\nu\right\rangle = \alpha\del{\bar\mu-\bar\nu}^2 \geq 0.
\end{equation*}
Thus Lasry--Lions monotonicity (Assumption~\ref{ass:LL-mono}) holds.

For each $x \in \Omega$,
\[
\abs{J(\mu,x) - J(\nu,x)} = \alpha\abs{\bar\mu - \bar \nu}\abs{x} \leq 2 \alpha\abs{\bar\mu - \bar \nu} \leq 2 \alpha W_1(\mu,\nu).
\]
The last inequality follows from Kantorovich--Rubinstein duality. Taking the supremum over \(x\in\Omega\) shows that \(W_1\)-Lipschitz continuity (Assumption~\ref{ass:W1-lip}) holds with \(L_m=2\alpha\).

Finally, \(\partial_x J(\mu,x)=4x(x^2-1)+\alpha\bar\mu\). Since \(x,\bar\mu\in[-2,2]\), \(\abs{\partial_x J(\mu,x)}\leq 24+2\alpha\). Thus, by the mean value theorem,
\[
\abs{J(\mu,x)-J(\mu,y)}
\leq (24+2\alpha)\abs{x-y},
\]
uniformly over \(\mu\in\cP(\Omega)\). Hence
Assumption~\ref{ass:spatial-lipschitz} holds with
\(L_x=24+2\alpha\).

The equilibrium is
\(\mu^{\ast}=\frac{1}{2}\delta_{-1}+\frac{1}{2}\delta_{1}\).

For this example, $D_\Omega=4$, $L_m=2\alpha$, and $L_x=24+2\alpha$. Hence the stepsize condition is $0<\lambda<1/(8\alpha)$. For \(M\) equidistant points on \([-2,2]\) and the uniform initialization, \(m_h=1/M\) and \(h=2/(M-1)\). Therefore, Theorem~\ref{thm:quantitative-mesh-approximate-MFE} gives the following upper bound on the Minty gap:
\begin{equation}
\label{eq:double-well-minty}
\sup_{\eta\in\cP(\Omega)}
\left\langle
J(\eta,\cdot),\bar\nu_{K,h}-\eta
\right\rangle
\leq
\frac{\log M}{\lambda K}
+\frac{48+12\alpha}{M-1}.
\end{equation}
Denote the right-hand side by \(\varepsilon_{K,M}\). If
\(\varepsilon_{K,M}\leq 16\alpha\), then Proposition~\ref{prop:eps-mvi-implies-delta-vi} gives the following upper bound on the VI gap:
\begin{equation}
\label{eq:double-well-vi-gap}
\sup_{\eta\in\cP(\Omega)}
\left\langle
J(\bar\nu_{K,h},\cdot),
\bar\nu_{K,h}-\eta
\right\rangle \leq
8\sqrt{\alpha\varepsilon_{K,M}} =
8\sqrt{
\alpha\left(
\frac{\log M}{\lambda K}
+\frac{48+12\alpha}{M-1}
\right)}.
\end{equation}
By Proposition~\ref{prop:eps-mfe-equivalent-eps-vi}, this is also the corresponding approximate-MFE bound.
For the numerical plots, we use \(\alpha=0.5\) and \(\lambda=0.24\).
Figure~\ref{fig:double-well-averaged-measure} shows the evolution of the averaged predictor measure, and Figure~\ref{fig:double-well-gaps} compares the two gaps with their theoretical bounds.

\begin{figure}
\centering
\includegraphics[width=\textwidth]{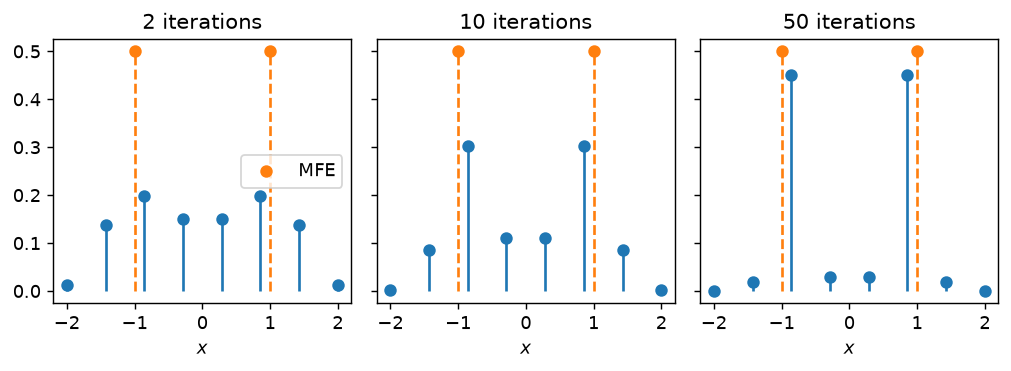}
\caption{Evolution of the averaged predictor measure \(\bar\nu_{K,h}\) for Example~\eqref{eq:double-well} at selected iterations on an \(M=8\) mesh, using \(\alpha=0.5\) and \(\lambda=0.24\). The dashed vertical lines indicate the support of the equilibrium $\mu^\ast=\frac12\delta_{-1}+\frac12\delta_1$.}
\label{fig:double-well-averaged-measure}
\end{figure}

\begin{figure}
\includegraphics[width = \textwidth]{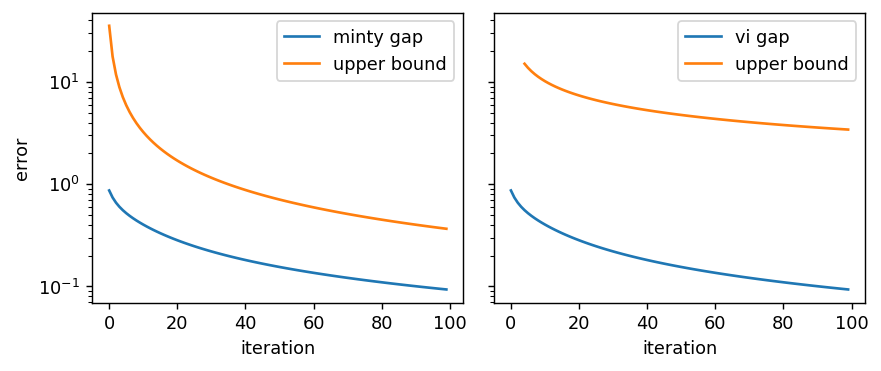}
\caption{Minty gap (left) and variational inequality gap (right), together with their respective upper bounds \eqref{eq:double-well-minty} and \eqref{eq:double-well-vi-gap}, for Example~\eqref{eq:double-well} with \(M=5001\), \(\alpha=0.5\), and \(\lambda=0.24\). Although the two gaps are of comparable magnitude, the upper bound for the VI gap is substantially looser.}
\label{fig:double-well-gaps}
\end{figure}

\subsection{Uniform Equilibrium Example}
Consider
\begin{equation}
\label{ex:uniform-equilibrium}
J(\mu,x) = \frac{x^2}{2} - \int_{\Omega}|x-y|\mu(dy),
\qquad \Omega=[-1,1].
\end{equation}
We verify that the assumptions are satisfied.

The set \(\Omega\) is compact. Moreover, for each
\(\mu\in\cP(\Omega)\), continuity of \(x\mapsto J(\mu,x)\) follows from dominated convergence. Thus
Assumption~\ref{ass:standing} holds.

Let $F_\mu(t):=\mu([-1,t])$ and $F_\nu(t):=\nu([-1,t])$. The standard one-dimensional identity for signed measures of total
mass zero gives
\[
-\iint_{\Omega^2}|x-y|(\mu-\nu)(dx)(\mu-\nu)(dy) = 2\int_{-1}^1 \bigl(F_\mu(t)-F_\nu(t)\bigr)^2\,dt.
\]
Consequently,
\begin{align*}
\left\langle
J(\mu,\cdot)-J(\nu,\cdot),\mu-\nu \right\rangle
&=
2\int_{-1}^1 \bigl(F_\mu(t)-F_\nu(t)\bigr)^2\,dt \\
&\geq
\left(\int_{-1}^1 |F_\mu(t)-F_\nu(t)|\,dt \right)^2 \\
&=
W_1(\mu,\nu)^2.
\end{align*}
The inequality follows from Cauchy--Schwarz, and the last equality is the
one-dimensional formula for \(W_1\). Thus strong Lasry--Lions monotonicity
(Assumption~\ref{ass:strong-LL-mono}) holds with \(\alpha=1\). In particular, Assumption~\ref{ass:LL-mono} is also satisfied.

Since \(y\mapsto |x-y|\) is \(1\)-Lipschitz, the
Kantorovich--Rubinstein duality gives
\[
\left|J(\mu,x)-J(\nu,x)\right| = \left| \int_\Omega |x-y|(\mu-\nu)(dy) \right| \leq W_1(\mu,\nu).
\]
Taking the supremum over \(x\in\Omega\) shows that
Assumption~\ref{ass:W1-lip} holds with \(L_m=1\).

Finally, for \(x,z\in\Omega\),
\[
|J(\mu,x)-J(\mu,z)| \leq \frac12|x^2-z^2| + \int_\Omega \bigl||x-y|-|z-y|\bigr|\mu(dy)\leq 2|x-z|.
\]
Hence Assumption~\ref{ass:spatial-lipschitz} holds with \(L_x=2\).
Since \(D_\Omega=2\), the stepsize condition becomes \(0<\lambda<\frac{1}{2}\).

The MFE is the uniform distribution on \([-1,1]\). For \(M\) equally spaced mesh points, the finite-mesh equilibrium \(\mu_h^\ast\) has full support, with \(m_h^\ast=\frac{1}{2(M-1)}\) and \(\delta_h=\frac{2}{M-1}\). Taking \(\lambda=0.49\), we have \(\gamma_h=\lambda/[2(M-1)^2]\) for \(M\geq 8\). Hence item~(3) of Theorem~\ref{thm:finite-mesh-last-iterate} gives an index \(k_{0,h}\) satisfying
\[
k_{0,h} \leq \left\lceil \frac{128D_{0,h}(M-1)^5}{\lambda} \right\rceil
\]
such that, for every \(k\geq k_{0,h}\),
\[
W_1(\mu_{k,h},\mu_h^\ast) \leq \frac{1}{2(M-1)} \left(1-\frac{\lambda}{4(M-1)^3}\right)^{(k-k_{0,h})/2}.
\]
For these parameters, we may choose \(k_{0,h}=123\). We illustrate the algorithm using \(M=8\), \(\lambda=0.49\), and the full-support initialization \(\mu_{0,h}(x_i)=\frac{1+0.9x_i}{M}\). Figure~\ref{fig:strong-lasry-lions-evolution} shows the convergence of the last iterates to the finite-mesh equilibrium, while Figure~\ref{fig:strong-lasry-lions-last-iterate} compares the observed last-iterate error with the geometric upper bound.
\begin{figure}[H]
\includegraphics[width = \textwidth]{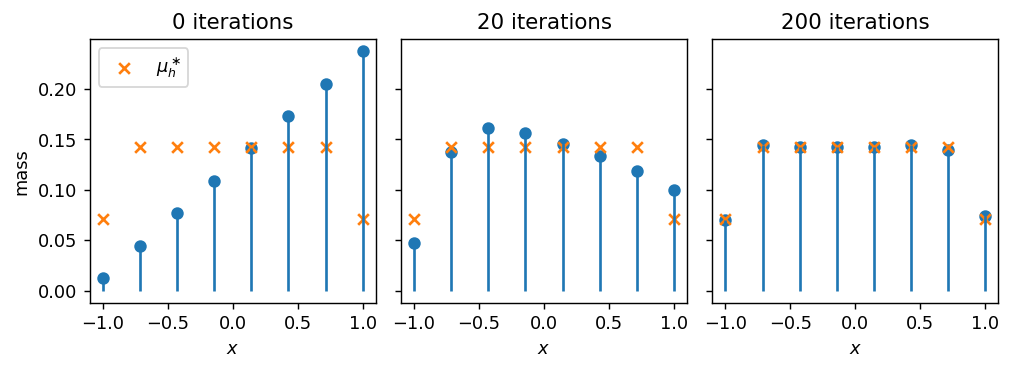}
\caption{Last iterates $\mu_{k,h}$ for Example~\eqref{ex:uniform-equilibrium} at selected iterations, using $M=8$ and $\lambda=0.49$. The crosses mark the finite-mesh equilibrium $\mu_h^\ast$.}
\label{fig:strong-lasry-lions-evolution}
\end{figure}
\begin{figure}[H]
\centering
\includegraphics{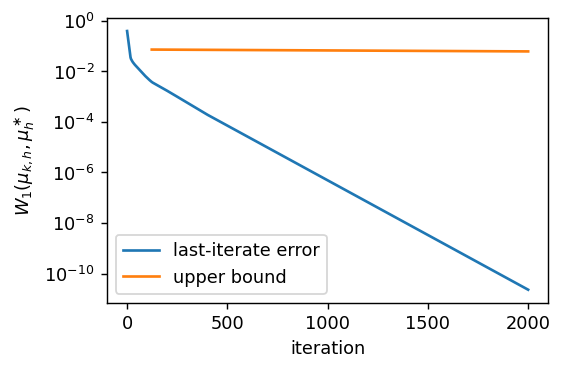}
\caption{Observed last-iterate error and the geometric upper bound from Theorem~\ref{thm:finite-mesh-last-iterate} for Example~\eqref{ex:uniform-equilibrium}, using $M=8$ and $\lambda=0.49$. The upper bound decreases very slowly over the displayed range and therefore appears nearly flat.}
\label{fig:strong-lasry-lions-last-iterate}
\end{figure}

\bibliographystyle{plain}
\bibliography{ref}

\appendix

\section{Proofs of Subsection~\ref{subsec:KL-estimate}}

\subsection{Proof of Lemma~\ref{lem:three-point}}
\begin{proof}
Since \(\Omega\) is compact and \(g\in C(\Omega)\), the function
\(e^{-\lambda g}\) is bounded, strictly positive, and continuous. Hence
\[
0<\int_\Omega e^{-\lambda g(x)}\,\rho(dx)<\infty.
\]
Therefore \(\zeta\) is a well-defined probability measure. Moreover,
\[
\frac{d\zeta}{d\rho}(x)
=
\frac{e^{-\lambda g(x)}}{\int_\Omega e^{-\lambda g(x)}\,\rho(dx)}
\]
is strictly positive and finite \(\rho\)-almost everywhere. Thus
\[
\zeta\sim\rho.
\]

We prove that \(\zeta\) is the unique minimizer. If
\(\eta\not\ll\rho\), then
\[
D_{\mathrm{KL}}(\eta\|\rho)=+\infty,
\]
so such \(\eta\) cannot minimize. Thus it suffices to consider
\(\eta\ll\rho\). For such \(\eta\), using
\[
\log\frac{d\zeta}{d\rho}
=
-\lambda g-\log Z_\rho,
\]
where $Z_\rho = \int_\Omega e^{-\lambda g(x)}\,\rho(dx)$, we get
\[
\log\frac{d\eta}{d\zeta}
=
\log\frac{d\eta}{d\rho}
-
\log\frac{d\zeta}{d\rho}
=
\log\frac{d\eta}{d\rho}
+
\lambda g
+
\log Z_\rho.
\]
Therefore
\[
\begin{aligned}
D_{\mathrm{KL}}(\eta\|\zeta)
&=
\int_\Omega
\log\frac{d\eta}{d\zeta}\,d\eta
\\
&=
\int_\Omega
\log\frac{d\eta}{d\rho}\,d\eta
+
\lambda\int_\Omega g\,d\eta
+
\log Z_\rho
\\
&=
D_{\mathrm{KL}}(\eta\|\rho)
+
\lambda\langle g,\eta\rangle
+
\log Z_\rho.
\end{aligned}
\]
Hence
\[
\lambda\langle g,\eta\rangle
+
D_{\mathrm{KL}}(\eta\|\rho)
=
D_{\mathrm{KL}}(\eta\|\zeta)
-
\log Z_\rho.
\]
Since
\[
D_{\mathrm{KL}}(\eta\|\zeta)\ge 0,
\]
with equality if and only if \(\eta=\zeta\), the unique minimizer is
\(\eta=\zeta\).

Now let \(\eta\in\mathcal P(\Omega)\) satisfy \(D_{\mathrm{KL}}(\eta\|\rho)<\infty\). Then \(\eta\ll\rho\). Since
\(\zeta\sim\rho\), we also have \(\eta\ll\zeta\). Decompose
\[
D_{\mathrm{KL}}(\eta\|\zeta)
=
D_{\mathrm{KL}}(\eta\|\rho)
+
\int_\Omega
\log\frac{d\rho}{d\zeta}\,d\eta.
\]
The first term is finite by assumption. The second term is finite because
\[
\log\frac{d\rho}{d\zeta}
=
\lambda g+\log Z_\rho
\in L^\infty(\rho),
\]
and \(\eta\in\mathcal P(\Omega)\). Therefore
\[
D_{\mathrm{KL}}(\eta\|\zeta)<\infty.
\]

Finally, we compute
\[
\begin{aligned}
&D_{\mathrm{KL}}(\eta\|\rho)
-
D_{\mathrm{KL}}(\eta\|\zeta)
-
D_{\mathrm{KL}}(\zeta\|\rho)
\\
&=
\int_\Omega
\log\frac{d\eta}{d\rho}\,d\eta
-
\int_\Omega
\log\frac{d\eta}{d\zeta}\,d\eta
-
\int_\Omega
\log\frac{d\zeta}{d\rho}\,d\zeta
\\
&=
\int_\Omega
\log\frac{d\zeta}{d\rho}\,d\eta
-
\int_\Omega
\log\frac{d\zeta}{d\rho}\,d\zeta
\\
&=
\int_\Omega
\log\frac{d\zeta}{d\rho}\,(d\eta-d\zeta).
\end{aligned}
\]
Using again
\[
\log\frac{d\zeta}{d\rho}
=
-\lambda g-\log Z_\rho,
\]
we obtain
\[
\int_\Omega
\log\frac{d\zeta}{d\rho}\,(d\eta-d\zeta)
=
-\lambda\int_\Omega g\,(d\eta-d\zeta)
-
\log Z_\rho
\int_\Omega (d\eta-d\zeta)
=
-\lambda\langle g,\eta-\zeta\rangle,
\]
because both \(\eta\) and \(\zeta\) are probability measures.
Thus
\[
D_{\mathrm{KL}}(\eta\|\rho)
-
D_{\mathrm{KL}}(\eta\|\zeta)
-
D_{\mathrm{KL}}(\zeta\|\rho)
=
\lambda\langle g,\zeta-\eta\rangle.
\]
\end{proof}

\subsection{Proof of Lemma~\ref{lem:one-step-Mirror-Prox}}
\begin{proof}
Finiteness of \(
D_{\mathrm{KL}}(\eta\|\mu_k)\) and \(D_{\mathrm{KL}}(\eta\|\nu_k)
\) follows directly from Lemma~\ref{lem:three-point}.

Apply Lemma~\ref{lem:three-point} to the second step with
\(\rho=\mu_k\), \(g=J(\nu_k,\cdot)\), and \(\zeta=\mu_{k+1}\). This gives
\[
\lambda\langle J(\nu_k,\cdot),\mu_{k+1}-\eta\rangle
=
D_{\mathrm{KL}}(\eta\|\mu_k)
-
D_{\mathrm{KL}}(\eta\|\mu_{k+1})
-
D_{\mathrm{KL}}(\mu_{k+1}\|\mu_k).
\]
Also apply Lemma~\ref{lem:three-point} to the first step, with
\(\rho=\mu_k\), \(g=J(\mu_k,\cdot)\), and \(\zeta=\nu_k\), taking
\(\eta=\mu_{k+1}\). Then
\[
\lambda\langle J(\mu_k,\cdot),\nu_k-\mu_{k+1}\rangle
=
D_{\mathrm{KL}}(\mu_{k+1}\|\mu_k)
-
D_{\mathrm{KL}}(\mu_{k+1}\|\nu_k)
-
D_{\mathrm{KL}}(\nu_k\|\mu_k).
\]
Therefore, combining the two equalities gives
\[
\begin{aligned}
\lambda\langle J(\nu_k,\cdot),\nu_k-\eta\rangle
&=
\lambda\langle J(\nu_k,\cdot),\mu_{k+1}-\eta\rangle
+
\lambda\langle J(\mu_k,\cdot),\nu_k-\mu_{k+1}\rangle \\
&\qquad
+
\lambda\langle J(\nu_k,\cdot)-J(\mu_k,\cdot),\nu_k-\mu_{k+1}\rangle\\
&=
D_{\mathrm{KL}}(\eta\|\mu_k)
-
D_{\mathrm{KL}}(\eta\|\mu_{k+1})
-
D_{\mathrm{KL}}(\mu_{k+1}\|\nu_k)
-
D_{\mathrm{KL}}(\nu_k\|\mu_k) \\
&\qquad
+
\lambda\langle J(\nu_k,\cdot)-J(\mu_k,\cdot),\nu_k-\mu_{k+1}\rangle.
\end{aligned}
\]
By the \(W_1\)-Lipschitz assumption, boundedness of $\Omega$, and Pinsker's inequality,
\[
\begin{aligned}
\lambda
\left|
\langle J(\nu_k,\cdot)-J(\mu_k,\cdot),\nu_k-\mu_{k+1}\rangle
\right|
&\le
\lambda
\|J(\nu_k,\cdot)-J(\mu_k,\cdot)\|_\infty
\|\nu_k-\mu_{k+1}\|_1 \\
&\le
\lambda L_m
W_1(\nu_k,\mu_k)
\|\nu_k-\mu_{k+1}\|_1 \\
&\le
\lambda L_m
D_\Omega\|\nu_k-\mu_k\|_1
\|\nu_k-\mu_{k+1}\|_1 \\ 
&\le
\lambda L_m D_\Omega
\sqrt{2D_{\mathrm{KL}}(\nu_k\|\mu_k)}
\sqrt{2D_{\mathrm{KL}}(\mu_{k+1}\|\nu_k)} \\
&\le
\lambda L_m D_\Omega
\left[
D_{\mathrm{KL}}(\nu_k\|\mu_k)
+
D_{\mathrm{KL}}(\mu_{k+1}\|\nu_k)
\right].
\end{aligned}
\]
Hence
\[
\lambda\langle J(\nu_k,\cdot),\nu_k-\eta\rangle
\le
D_{\mathrm{KL}}(\eta\|\mu_k)
-
D_{\mathrm{KL}}(\eta\|\mu_{k+1})
-
(1-\lambda L_m D_\Omega)
\left[
D_{\mathrm{KL}}(\nu_k\|\mu_k)
+
D_{\mathrm{KL}}(\mu_{k+1}\|\nu_k)
\right].
\]
\end{proof}

\section{Proofs of Subsection~\ref{subsec:ergodic-minty}}
\subsection{Proof of Lemma~\ref{lem:ergodic-minty}}
\begin{proof}
Since \(\lambda L_mD_\Omega<1\), Lemma~\ref{lem:one-step-Mirror-Prox} gives
\[
\lambda\langle J(\nu_k,\cdot),\nu_k-\eta\rangle
\le
D_{\mathrm{KL}}(\eta\|\mu_k)
-
D_{\mathrm{KL}}(\eta\|\mu_{k+1}).
\]
Summing from \(k=0\) to \(K-1\), we get
\[
\lambda
\sum_{k=0}^{K-1}
\langle J(\nu_k,\cdot),\nu_k-\eta\rangle
\le
D_{\mathrm{KL}}(\eta\|\mu_0)
-
D_{\mathrm{KL}}(\eta\|\mu_K)
\le
D_{\mathrm{KL}}(\eta\|\mu_0).
\]
Hence
\[
\frac1K
\sum_{k=0}^{K-1}
\langle J(\nu_k,\cdot),\nu_k-\eta\rangle
\le
\frac{D_{\mathrm{KL}}(\eta\|\mu_0)}{\lambda K}.
\]
By Lasry--Lions monotonicity,
\[
\langle J(\nu_k,\cdot)-J(\eta,\cdot),\nu_k-\eta\rangle\ge 0.
\]
Therefore
\[
\langle J(\eta,\cdot),\nu_k-\eta\rangle
\le
\langle J(\nu_k,\cdot),\nu_k-\eta\rangle.
\]
Averaging in \(k\), we obtain
\[
\langle J(\eta,\cdot),\bar\nu_K-\eta\rangle
\le
\frac1K
\sum_{k=0}^{K-1}
\langle J(\nu_k,\cdot),\nu_k-\eta\rangle
\le
\frac{D_{\mathrm{KL}}(\eta\|\mu_0)}{\lambda K}.
\]

\end{proof}

\subsection{Proof of Corollary~\ref{cor:cluster-point-mintyVI}}
\begin{proof}
Since \(\Omega\) is compact, \(\mathcal P(\Omega)\) is compact for the topology
of weak convergence. Hence the sequence
\((\bar\nu_K)_{K\ge 1}\subset \mathcal P(\Omega)\) admits a weakly convergent
subsequence. Therefore there exist \(K_j\to\infty\) and
\(\bar\nu\in\mathcal P(\Omega)\) such that
\[
\bar\nu_{K_j}\Rightarrow \bar\nu.
\]

Now fix any \(\eta\in\mathcal P(\Omega)\) such that
\[
D_{\mathrm{KL}}(\eta\|\mu_0)<\infty.
\]
Applying Lemma~\ref{lem:ergodic-minty} along the subsequence \(K_j\), we obtain
\[
\langle J(\eta,\cdot),\bar\nu_{K_j}-\eta\rangle
\le
\frac{D_{\mathrm{KL}}(\eta\|\mu_0)}{\lambda K_j}.
\]
Since \(K_j\to\infty\), the right-hand side converges to \(0\).

On the other hand, \(J(\eta,\cdot)\in C(\Omega)\). Hence weak convergence gives
\[
\int_\Omega J(\eta,x)\,\bar\nu_{K_j}(dx)
\to
\int_\Omega J(\eta,x)\,\bar\nu(dx).
\]
Therefore
\[
\langle J(\eta,\cdot),\bar\nu_{K_j}-\eta\rangle
\to
\langle J(\eta,\cdot),\bar\nu-\eta\rangle.
\]
Passing to the limit \(j\to\infty\) yields
\[
\langle J(\eta,\cdot),\bar\nu-\eta\rangle\le 0.
\]
Since \(\eta\) was arbitrary among all measures satisfying
\(D_{\mathrm{KL}}(\eta\|\mu_0)<\infty\), the claim follows.
\end{proof}

\subsection{Proof of Corollary~\ref{cor:finite-mesh-approximate-MFE}}
\begin{proof}
Let \(\eta\in\mathcal P(\Omega)\). Since \(\mu_0\) has full
support,
\[
\begin{aligned}
    D_{\mathrm{KL}}(\eta\|\mu_0)
    &=
    \sum_{x\in \Omega}
        \eta(\{x\})
        \log
        \frac{\eta(\{x\})}{\mu_0(\{x\})} \\
    &=
    \sum_{x\in \Omega}
        \eta(\{x\})\log\eta(\{x\})
    -
    \sum_{x\in \Omega}
        \eta(\{x\})\log\mu_0(\{x\}) \\
    &\leq
    -\sum_{x\in \Omega}
        \eta(\{x\})\log m
    =
    \log\frac1{m}.
\end{aligned}
\]
Here we use the convention \(0\log0=0\) and the fact that
\(\sum_x\eta(\{x\})\log\eta(\{x\})\leq0\).

Applying Lemma~\ref{lem:ergodic-minty} on the finite-state space
\(\Omega\) gives
\[
    \bigl\langle
        J(\eta,\cdot),
        \bar\nu_K-\eta
    \bigr\rangle
    \leq
    \frac{
        D_{\mathrm{KL}}(\eta\|\mu_0)
    }{
        \lambda K
    }
    \leq
    \frac{\log(1/m)}{\lambda K}
    =\varepsilon_K,\qquad
    \forall \eta\in\mathcal P(\Omega).
\]

\end{proof}

\section{Proofs of Subsection~\ref{subsec:finite-mesh-approximation}}
\subsection{Proof of Proposition~\ref{prop:fixed-K-consistency}}
We first introduce notation for the exponential-tilt map.
\begin{definition}
For
\(\rho\in\mathcal P(\Omega)\) and \(g\in C(\Omega)\), define
\[
\mathcal T_\lambda(\rho,g)(dx)
:=
\frac{e^{-\lambda g(x)}\rho(dx)}
{\int_\Omega e^{-\lambda g(z)}\rho(dz)}.
\]
Since \(\Omega\) is compact and \(g\in C(\Omega)\), the denominator is
strictly positive and finite.
\end{definition}
In this notation, the KL Mirror-Prox iteration becomes \[ \nu_k = \mathcal T_\lambda \bigl(\mu_k,J(\mu_k,\cdot)\bigr), \qquad \mu_{k+1} = \mathcal T_\lambda \bigl(\mu_k,J(\nu_k,\cdot)\bigr). \]

\begin{proof}
We first prove the basic stability property of the exponential-tilt map.

Suppose \(\rho_h\Rightarrow \rho\) in \(\mathcal P(\Omega)\) and
\(g_h\to g\) uniformly on \(\Omega\). We claim that
\[
\mathcal T_\lambda(\rho_h,g_h)
\Rightarrow
\mathcal T_\lambda(\rho,g).
\]
Let \(\varphi\in C(\Omega)\). Then
\[
\int_\Omega \varphi(x)\,
\mathcal T_\lambda(\rho_h,g_h)(dx)
=
\frac{
\int_\Omega \varphi(x)e^{-\lambda g_h(x)}\rho_h(dx)
}{
\int_\Omega e^{-\lambda g_h(x)}\rho_h(dx)
}.
\]
Since \(g_h\to g\) uniformly and \(\varphi\in C(\Omega)\), we have
\[
\varphi e^{-\lambda g_h}
\to
\varphi e^{-\lambda g}
\quad\text{uniformly on }\Omega,
\]
and
\[
e^{-\lambda g_h}
\to
e^{-\lambda g}
\quad\text{uniformly on }\Omega.
\]
Using weak convergence of \(\rho_h\), we obtain
\[
\int_\Omega \varphi e^{-\lambda g_h}\,d\rho_h
\to
\int_\Omega \varphi e^{-\lambda g}\,d\rho,
\]
and
\[
\int_\Omega e^{-\lambda g_h}\,d\rho_h
\to
\int_\Omega e^{-\lambda g}\,d\rho.
\]
The limiting denominator is strictly positive. Therefore
\[
\int_\Omega \varphi\,d\mathcal T_\lambda(\rho_h,g_h)
\to
\int_\Omega \varphi\,d\mathcal T_\lambda(\rho,g).
\]
Hence
\[
\mathcal T_\lambda(\rho_h,g_h)
\Rightarrow
\mathcal T_\lambda(\rho,g).
\]

We now prove the result by induction on \(k\).

For \(k=0\), the convergence \(\mu_{0,h}\Rightarrow\mu_0\) holds by
assumption. Since $\Omega$ is compact, $W_1$ metrizes weak convergence on $\mathcal{P}(\Omega)$. Hence, by Assumption~\ref{ass:W1-lip},
\[
J(\mu_{0,h},\cdot)\to J(\mu_0,\cdot)
\quad\text{uniformly on }\Omega.
\]
Applying the stability property of \(\mathcal T_\lambda\), we get
\[
\nu_{0,h}
=
\mathcal T_\lambda(\mu_{0,h},J(\mu_{0,h},\cdot))
\Rightarrow
\mathcal T_\lambda(\mu_0,J(\mu_0,\cdot))
=
\nu_0.
\]
Again by Assumption~\ref{ass:W1-lip},
\[
J(\nu_{0,h},\cdot)\to J(\nu_0,\cdot)
\quad\text{uniformly on }\Omega.
\]
Therefore
\[
\mu_{1,h}
=
\mathcal T_\lambda(\mu_{0,h},J(\nu_{0,h},\cdot))
\Rightarrow
\mathcal T_\lambda(\mu_0,J(\nu_0,\cdot))
=
\mu_1.
\]

Now assume that
\[
\mu_{k,h}\Rightarrow\mu_k.
\]
Assumption~\ref{ass:W1-lip} gives
\[
J(\mu_{k,h},\cdot)\to J(\mu_k,\cdot)
\quad\text{uniformly on }\Omega.
\]
Hence
\[
\nu_{k,h}
=
\mathcal T_\lambda(\mu_{k,h},J(\mu_{k,h},\cdot))
\Rightarrow
\mathcal T_\lambda(\mu_k,J(\mu_k,\cdot))
=
\nu_k.
\]
Then
\[
J(\nu_{k,h},\cdot)\to J(\nu_k,\cdot)
\quad\text{uniformly on }\Omega,
\]
and so
\[
\mu_{k+1,h}
=
\mathcal T_\lambda(\mu_{k,h},J(\nu_{k,h},\cdot))
\Rightarrow
\mathcal T_\lambda(\mu_k,J(\nu_k,\cdot))
=
\mu_{k+1}.
\]
This proves the induction.

Finally, for any \(\varphi\in C(\Omega)\),
\[
\int_\Omega \varphi\,d\bar\nu_{K,h}
=
\frac1K\sum_{k=0}^{K-1}
\int_\Omega \varphi\,d\nu_{k,h}
\to
\frac1K\sum_{k=0}^{K-1}
\int_\Omega \varphi\,d\nu_k
=
\int_\Omega \varphi\,d\bar\nu_K.
\]
Thus
\[
\bar\nu_{K,h}\Rightarrow\bar\nu_K.
\]
\end{proof}

\subsection{Proof of Theorem~\ref{thm:mesh-refinement-MFE}}
\begin{proof}
Fix \(h>0\). Let \(\eta\in\mathcal P(\Omega)\) be arbitrary and define its mesh
projection by
\[
\eta_h:=(Q_h)_\#\eta\in\mathcal P(X_h).
\]
Since \(|Q_h(x)-x|\le h\), we have
\[
W_1(\eta_h,\eta)\le h.
\]
In particular,
\[\eta_{h}\Rightarrow\eta.\]

Applying Lemma~\ref{lem:ergodic-minty} to \(\mathcal P(X_h)\), we obtain, for every \(K\ge1\), 
\[ \langle J(\eta_h,\cdot), \bar\nu_{K,h}-\eta_h \rangle \le \frac{D_{\mathrm{KL}}(\eta_h\|\mu_{0,h})}{\lambda K}. \] 
Because \(\mu_{0,h}\) has full support on \(X_h\), this relative entropy is finite. 
Moreover, using 
\[ \sum_{x\in X_h} \eta_h(\{x\})\log\eta_h(\{x\}) \le0 \] 
and 
\[ \mu_{0,h}(\{x\})\ge m_h, \qquad \forall x\in X_h, \] 
we obtain 
\[ 
D_{\mathrm{KL}}(\eta_h\|\mu_{0,h}) 
= \sum_{x\in X_h} \eta_h(\{x\})\log\eta_h(\{x\}) - \sum_{x\in X_h} \eta_h(\{x\})\log\mu_{0,h}(\{x\}) 
\le \log\frac1{m_h}.
\] 
Therefore, 
\begin{equation*}
\langle J(\eta_h,\cdot), \eta_h-\bar\nu_{K,h} \rangle \ge -\frac{\log(1/m_h)}{\lambda K}, \qquad \forall K\ge1. 
\end{equation*}

Since \(\Omega\) is compact, \(\mathcal P(\Omega)\) is weakly compact. Hence
every sequence \((\bar\nu_h)\) admits weakly convergent subsequences. Let
\[
\bar\nu_{h_j}\Rightarrow\bar\nu
\]
along some sequence \(h_j\downarrow0\). We pass to the limit in
\[
\langle J(\eta_{h_j},\cdot),\bar\nu_{h_j}-\eta_{h_j}\rangle
\le
\frac{\log(1/m_{h_j})}{\lambda K_{h_j}}.
\]

Since \(W_1(\eta_{h_j},\eta)\le h_j\), Assumption~\ref{ass:W1-lip}
implies
\[
\|J(\eta_{h_j},\cdot)-J(\eta,\cdot)\|_{L^\infty(\Omega)}
\to0.
\]
Also, $\eta_{h_j}\Rightarrow\eta$ and \(J(\eta,\cdot)\in C(\Omega)\). Hence
\[
\int_\Omega J(\eta_{h_j},x)\,\bar\nu_{h_j}(dx)
\to
\int_\Omega J(\eta,x)\,\bar\nu(dx),
\]
and similarly,
\[
\int_\Omega J(\eta_{h_j},x)\,\eta_{h_j}(dx)
\to
\int_\Omega J(\eta,x)\,\eta(dx).
\]
In addition, by the choice of $K_h$,
\[
\frac{\log(1/m_{h_j})}{K_{h_j}}\to0.
\]
Thus
\[
\langle J(\eta,\cdot),\bar\nu-\eta\rangle\le0.
\]
Since \(\eta\in\mathcal P(\Omega)\) was arbitrary, \(\bar\nu\) is a Minty
solution:
\[
\langle J(\eta,\cdot),\bar\nu-\eta\rangle\le0,
\qquad
\forall \eta\in\mathcal P(\Omega).
\]
By Proposition~\ref{prop:mvi-implies-vi}, \(\bar\nu\) solves the variational inequality.

Finally, suppose \(\bar\nu\) is the unique MFE. If \(\bar\nu_h\) did not converge weakly to \(\bar\nu\), then
there would exist a subsequence \((\bar\nu_{h_j})\) that stays away from
\(\bar\nu\). By compactness of \(\mathcal P(\Omega)\), this subsequence has a
further weakly convergent subsequence
\(
\bar\nu_{h_{j_\ell}}\Rightarrow\bar\nu'.
\)
By the result already proved, \(\bar\nu'\) is an MFE, so
\(
\bar\nu'=\bar\nu
\)
by uniqueness,
which contradicts the choice of the subsequence. Hence
\(
\bar\nu_h\Rightarrow\bar\nu.
\)
\end{proof}

\subsection{Proof of Theorem~\ref{thm:quantitative-mesh-approximate-MFE}}

\begin{proof} 
Fix \(h>0\), \(K\ge1\), and \(\eta\in\mathcal P(\Omega)\), and define \[ \eta_h:=(Q_h)_\#\eta. \] 
As shown in the proof of Theorem~\ref{thm:mesh-refinement-MFE}, 
\[ W_1(\eta,\eta_h)\le h, \] 
and 
\begin{equation} \label{eq:quantitative-prop-discrete-bound}
\bigl\langle J(\eta_h,\cdot), \eta_h-\bar\nu_{K,h} \bigr\rangle \ge -\frac{\log(1/m_h)}{\lambda K}. 
\end{equation} 
We have 
\[ \begin{aligned} \langle J(\eta,\cdot),\eta-\bar\nu_{K,h}\rangle 
={}& \langle J(\eta_h,\cdot),\eta_h-\bar\nu_{K,h}\rangle \\ 
&+ \langle J(\eta,\cdot)-J(\eta_h,\cdot),\eta-\bar\nu_{K,h}\rangle \\
&- \langle J(\eta_h,\cdot),\eta_h-\eta\rangle. \end{aligned} \]
For the second term, by Assumption~\ref{ass:W1-lip}, 
\[ \left| \langle J(\eta,\cdot)-J(\eta_h,\cdot),\eta-\bar\nu_{K,h}\rangle \right|
\le \|J(\eta,\cdot)-J(\eta_h,\cdot)\|_{\infty} 
\|\eta-\bar\nu_{K,h}\|_1
\le L_mW_1(\eta,\eta_h)\cdot 2 \le 2L_mh. \] 
For the last term, by Assumption~\ref{ass:spatial-lipschitz},
\[ \left| \langle J(\eta_h,\cdot),\eta_h-\eta\rangle \right| \le L_x W_1(\eta,\eta_h) \le L_xh. \]
Combining these estimates with \eqref{eq:quantitative-prop-discrete-bound} yields 
\[ \bigl\langle J(\eta,\cdot), \eta-\bar\nu_{K,h} \bigr\rangle 
\ge -\frac{\log(1/m_h)}{\lambda K} -(2L_m+L_x)h
= -\varepsilon_{K,h}.\] 
Since \(\eta\in\mathcal P(\Omega)\) was arbitrary, the result follows. 
\end{proof}

\section{Proofs of Subsection~\ref{subsec:eps-MFE}}
\subsection{Proof of Proposition~\ref{prop:eps-mfe-equivalent-eps-vi}}
\begin{proof}
Assume first that \(\mu^\varepsilon\) satisfies the \(\varepsilon\)-VI. Since
the inequality holds for every \(\eta\in\mathcal P(\Omega)\), we may choose
\(\eta=\delta_x\) for arbitrary \(x\in\Omega\). Then
\[
    J(\mu^\varepsilon,x)
    -
    \int_\Omega J(\mu^\varepsilon,y)\,\mu^\varepsilon(dy)
    \ge
    -\varepsilon.
\]
Equivalently,
\[
    \int_\Omega J(\mu^\varepsilon,y)\,\mu^\varepsilon(dy)
    \le
    J(\mu^\varepsilon,x)+\varepsilon,
    \qquad
    \forall x\in\Omega.
\]
Taking the minimum over \(x\in\Omega\), which exists because
\(J(\mu^\varepsilon,\cdot)\in C(\Omega)\) and \(\Omega\) is compact, gives
\[
    \int_\Omega J(\mu^\varepsilon,y)\,\mu^\varepsilon(dy)
    \le
    \min_{x\in\Omega}J(\mu^\varepsilon,x)+\varepsilon.
\]
Thus \(\mu^\varepsilon\) is an \(\varepsilon\)-MFE.

Conversely, assume that \(\mu^\varepsilon\) is an \(\varepsilon\)-MFE. Then
\[
    \int_\Omega J(\mu^\varepsilon,y)\,\mu^\varepsilon(dy)
    \le
    \min_{x\in\Omega}J(\mu^\varepsilon,x)+\varepsilon.
\]
For arbitrary \(\eta\in\mathcal P(\Omega)\),
\[
    \min_{x\in\Omega}J(\mu^\varepsilon,x)
    \le
    \int_\Omega J(\mu^\varepsilon,y)\,\eta(dy).
\]
Therefore
\[
    \int_\Omega J(\mu^\varepsilon,y)\,\mu^\varepsilon(dy)
    \le
    \int_\Omega J(\mu^\varepsilon,y)\,\eta(dy)+\varepsilon.
\]
Rearranging gives
\[
    \langle J(\mu^\varepsilon,\cdot),\eta-\mu^\varepsilon\rangle
    \ge
    -\varepsilon.
\]
Since \(\eta\in\mathcal P(\Omega)\) was arbitrary, \(\mu^\varepsilon\)
satisfies the \(\varepsilon\)-VI.
\end{proof}

\subsection{Proof of Proposition~\ref{prop:mvi-implies-vi}}
\begin{proof}
Fix \(\eta\in\mathcal P(\Omega)\). For \(t\in(0,1]\), set
\[
    \rho_t:=(1-t)\mu+t\eta.
\]
Since \(\mathcal P(\Omega)\) is convex, \(\rho_t\in\mathcal P(\Omega)\).
Applying the Minty VI with \(\rho_t\) gives
\[
    \langle J(\rho_t,\cdot),\rho_t-\mu\rangle\ge0.
\]
Since
\[
    \rho_t-\mu=t(\eta-\mu),
\]
we get
\[
    \langle J(\rho_t,\cdot),\eta-\mu\rangle\ge0.
\]
By the continuity assumption, letting \(t\downarrow0\) gives
\[
    \langle J(\mu,\cdot),\eta-\mu\rangle\ge0.
\]
Since \(\eta\) was arbitrary, \(\mu\) satisfies the exact VI.
\end{proof}

\subsection{Proof of Proposition~\ref{prop:eps-vi-implies-eps-mvi}}
\begin{proof}
Fix \(\eta\in\mathcal P(\Omega)\). By Lasry--Lions monotonicity,
\[
    \langle J(\eta,\cdot)-J(\mu^\varepsilon,\cdot),
    \eta-\mu^\varepsilon\rangle
    \ge0.
\]
Therefore
\[
\begin{aligned}
    \langle J(\eta,\cdot),\eta-\mu^\varepsilon\rangle
    =
    \langle J(\mu^\varepsilon,\cdot),\eta-\mu^\varepsilon\rangle
    +
    \langle J(\eta,\cdot)-J(\mu^\varepsilon,\cdot),
    \eta-\mu^\varepsilon\rangle
    \ge
    -\varepsilon.
\end{aligned}
\]
Since \(\eta\) was arbitrary, the claim follows.
\end{proof}

\subsection{Proof of Proposition~\ref{prop:eps-mvi-implies-delta-vi}}
\begin{proof}
We follow the main argument of
\cite[Proposition~3.2(i)]{bigi2023approximate},
with modifications needed to accommodate the measure-space setting.

Suppose, for contradiction, that \(\mu^\varepsilon\) does not satisfy the
\(\delta\)-VI. Then there exists \(\widetilde\eta\in\mathcal P(\Omega)\) such
that
\[
    \langle J(\mu^\varepsilon,\cdot),
    \widetilde\eta-\mu^\varepsilon\rangle
    <
    -\delta.
\]
Set
\[
    t_0:=\left(\frac{\varepsilon}{2L_mD_\Omega}\right)^{1/2}\in(0,1],
    \qquad
    \rho_{t_0}:=(1-t_0)\mu^\varepsilon+t_0\widetilde\eta.
\]
Since \(\mathcal P(\Omega)\) is convex, \(\rho_{t_0}\in\mathcal P(\Omega)\).
Applying the \(\varepsilon\)-Minty VI with \(\eta=\rho_{t_0}\), we get
\[
    \langle J(\rho_{t_0},\cdot),\rho_{t_0}-\mu^\varepsilon\rangle
    \ge
    -\varepsilon.
\]
Since
\[
    \rho_{t_0}-\mu^\varepsilon
    =
    t_0(\widetilde\eta-\mu^\varepsilon),
\]
this becomes
\[
    t_0
    \langle J(\rho_{t_0},\cdot),
    \widetilde\eta-\mu^\varepsilon\rangle
    \ge
    -\varepsilon.
\]
Therefore
\[
\begin{aligned}
-\varepsilon
&\le
t_0
\langle J(\rho_{t_0},\cdot),
\widetilde\eta-\mu^\varepsilon\rangle
=
t_0
\langle J(\mu^\varepsilon,\cdot),
\widetilde\eta-\mu^\varepsilon\rangle
+
t_0
\langle J(\rho_{t_0},\cdot)-J(\mu^\varepsilon,\cdot),
\widetilde\eta-\mu^\varepsilon\rangle.
\end{aligned}
\]
Using the strict violation and Assumption~\ref{ass:W1-lip}, we obtain
\[
\begin{aligned}
-\varepsilon
&<
-\delta t_0
+
t_0
\left|
\langle J(\rho_{t_0},\cdot)-J(\mu^\varepsilon,\cdot),
\widetilde\eta-\mu^\varepsilon\rangle
\right| \\
&\le
-\delta t_0
+
2t_0
\|J(\rho_{t_0},\cdot)-J(\mu^\varepsilon,\cdot)\|_\infty \\
&\le
-\delta t_0
+
2L_m t_0 W_1(\rho_{t_0},\mu^\varepsilon).
\end{aligned}
\]
Moreover,
\[
    W_1(\rho_{t_0},\mu^\varepsilon)
    \le
    t_0 W_1(\widetilde\eta,\mu^\varepsilon)
    \le
    t_0D_\Omega.
\]
Hence
\[
    -\varepsilon
    <
    -\delta t_0
    +
    2L_mD_\Omega t_0^2.
\]
By the definition of \(t_0\),
\[
    2L_mD_\Omega t_0^2=\varepsilon.
\]
Also, by the definition of \(\delta\),
\[
    \delta t_0=2\varepsilon.
\]
Therefore
\[
    -\varepsilon
    <
    -2\varepsilon+\varepsilon
    =
    -\varepsilon,
\]
which is impossible. Hence \(\mu^\varepsilon\) satisfies the \(\delta\)-VI.
\end{proof}

\section{Proofs of Section~\ref{sec:strong-LL-mono}}
\subsection{Proof of Theorem~\ref{thm:finite-mesh-last-iterate}}

We first prove a lemma. 
The second estimate below is a local version of the standard
finite-alphabet reverse Pinsker inequality
\cite[Lemma~6.3]{csiszar2006context};
see also \cite{sason2015upper} for sharper finite-alphabet bounds.
The first estimate extends this type of control to the case where
$\mu^\ast$ need not have full support, at the price of a linear rather
than quadratic dependence on $\|\mu-\mu^\ast\|_1$.

\begin{lemma}
\label{lem:local-reverse-kl-finite-grid}
Let
\(
    \mathcal X=\{x_1,\dots,x_n\}.
\)
Let
\(
    \mu^\ast=\sum_{i=1}^n p_i^\ast\delta_{x_i}
    \in \mathcal P(\mathcal X),
\)
and define \(S_\ast:=\{i:p_i^\ast>0\}\) and \(m_\ast:=\min_{i\in S_\ast}p_i^\ast>0\).
Then the following two estimates hold.

\begin{enumerate}
    \item
For every \(\mu\in\mathcal P(\mathcal X)\) with full support satisfying
\[
    \|\mu-\mu^\ast\|_1\le \frac{m_\ast}{2},
\]
one has
\[
    D_{\rm KL}(\mu^\ast\|\mu)
    \le
    2\|\mu-\mu^\ast\|_1.
\]

\item 
If, in addition, \(\mu^\ast\) has full support on \(\mathcal X\), namely
\[
    m_\ast=\min_{1\le i\le n}p_i^\ast>0,
\]
then, for every \(\mu\in\mathcal P(\mathcal X)\) with full support satisfying
\[
    \|\mu-\mu^\ast\|_1\le \frac{m_\ast}{2},
\]
one has the sharper quadratic estimate
\[
    D_{\rm KL}(\mu^\ast\|\mu)
    \le
    \frac{1}{m_\ast}
    \|\mu-\mu^\ast\|_1^2.
\]

\end{enumerate}

\end{lemma}

\begin{proof}
Let
\[
    \mu=\sum_{i=1}^n p_i\delta_{x_i}\in\mathcal P(\mathcal X)
\]
with $p_i>0$ for all $i$.
Using \(\log u\le u-1\) and \( \sum_{i=1}^n
    p_i^\ast= \sum_{i=1}^n
    p_i=1\), we get 
\[
    D_{\rm KL}(p^\ast\|p)
    \le
    \sum_{i=1}^n
    p_i^\ast \left(\frac{p_i^\ast}{p_i}-1\right)
    =
    \sum_{i=1}^n
    \frac{(p_i^\ast-p_i)^2}{p_i}.
\]

Let
\[
    r:=\|\mu-\mu^\ast\|_1.
\]
If \(r\le m_\ast/2\), then for every \(i\in S_\ast\),
\[
    p_i\ge p_i^\ast-r\ge \frac{m_\ast}{2}.
\]
Therefore,
\[
\begin{aligned}
    D_{\rm KL}(\mu^\ast\|\mu)
    &\le
    \sum_{i\in S_\ast}
    \frac{(p_i^\ast-p_i)^2}{p_i}
    +
    \sum_{i\notin S_\ast}
    \frac{p_i^2}{p_i}  \\
    &\le
    \frac{2}{m_\ast}
    \sum_{i\in S_\ast}(p_i^\ast-p_i)^2
    +
    \sum_{i\notin S_\ast}p_i  \\
    &\le
    \frac{2}{m_\ast}r^2+r.
\end{aligned}
\]
If \(r\le m_\ast/2\), this implies
\[
    D_{\rm KL}(\mu^\ast\|\mu)
    \le
    2r.
\]

If \(\mu^\ast\) has full support, then \(S_\ast=\{1,\dots,n\}\) and 
\(p_i\ge \frac{m_\ast}{2}\) holds for all $i$.
Hence
\[
    D_{\rm KL}(\mu^\ast\|\mu)
    \le
    \sum_{i=1}^n
    \frac{(p_i^\ast-p_i)^2}{p_i}
    \le
    \frac{2}{m_\ast}
    \sum_{i=1}^n(p_i^\ast-p_i)^2
    \le
    \frac{1}{m_\ast}
    \left(\sum_{i=1}^n |p_i^\ast-p_i|\right)^2
    =
    \frac{1}{m_\ast}\|\mu-\mu^\ast\|_1^2,
\]
where the penultimate inequality uses
\(\sum_{i=1}^n(p_i^\ast-p_i)=0\).
\end{proof}

\begin{proof}[{\bf Proof of Theorem~\ref{thm:finite-mesh-last-iterate}}]
Existence of \(\mu_h^\ast\) follows from
Corollary~\ref{cor:cluster-point-mintyVI} and
Proposition~\ref{prop:mvi-implies-vi}. Uniqueness follows from strong
Lasry--Lions monotonicity.

Applying Lemma~\ref{lem:one-step-Mirror-Prox} with
\(\eta_h=\mu_h^\ast\), we get
\begin{equation}\label{eq:original-estimate}
    \lambda
    \langle J(\nu_{k,h},\cdot),\nu_{k,h}-\mu_h^\ast\rangle
    \le
    D_{k,h}-D_{k+1,h}
    -
    (1-\lambda L_mD_\Omega)
    \left[
        D_{\rm KL}(\nu_{k,h}\|\mu_{k,h})
        +
        D_{\rm KL}(\mu_{k+1,h}\|\nu_{k,h})
    \right].
\end{equation}
Since \(\mu_h^\ast\) solves the finite-mesh VI,
\[
    \langle J(\mu_h^\ast,\cdot),\nu_{k,h}-\mu_h^\ast\rangle\ge0.
\]
Combining this inequality with \(W_1\)-strong monotonicity yields
\[
    \langle J(\nu_{k,h},\cdot),\nu_{k,h}-\mu_h^\ast\rangle
    \ge
    \alpha W_1(\nu_{k,h},\mu_h^\ast)^2.
\]
Substituting this into \eqref{eq:original-estimate} and dropping the nonnegative term
\(D_{\rm KL}(\mu_{k+1,h}\|\nu_{k,h})\), we obtain
\begin{equation}\label{eq:estimate-after-mono}
    \lambda\alpha W_1(\nu_{k,h},\mu_h^\ast)^2
    +
    (1-\lambda L_mD_\Omega) D_{\rm KL}(\nu_{k,h}\|\mu_{k,h})
    \le
    D_{k,h}-D_{k+1,h}.
\end{equation}
Also, Pinsker's inequality and
\[
    W_1(\nu_{k,h},\mu_{k,h})
    \le
    D_\Omega\|\nu_{k,h}-\mu_{k,h}\|_1
\]
give
\[
    D_{\rm KL}(\nu_{k,h}\|\mu_{k,h})
    \ge
    \frac{1}{2D_\Omega^2}
    W_1(\nu_{k,h},\mu_{k,h})^2.
\]
Hence, \eqref{eq:estimate-after-mono} becomes
\[
\begin{aligned}
    \lambda\alpha W_1(\nu_{k,h},\mu_h^\ast)^2
    +
    \frac{1-\lambda L_mD_\Omega}{2D_\Omega^2}
    W_1(\nu_{k,h},\mu_{k,h})^2
    \le
    D_{k,h}-D_{k+1,h}.
\end{aligned}
\]
Since
\[
    W_1(\mu_{k,h},\mu_h^\ast)^2
    \le
    2W_1(\nu_{k,h},\mu_h^\ast)^2
    +
    2W_1(\nu_{k,h},\mu_{k,h})^2,
\]
we obtain \eqref{eq:finite-mesh-W1-descent}. In particular,
\(D_{k,h}\) is nonincreasing, and
\[
    W_1(\mu_{k,h},\mu_h^\ast)\to0.
\]

Next, we prove the last-iterate estimates. On the finite mesh,
\[
    W_1(\rho,\sigma)
    \ge
    \frac12
    \delta_h
    \|\rho-\sigma\|_1,
    \qquad
    \rho,\sigma\in\mathcal P(X_h).
\]
Using this estimate and Pinsker's inequality in \eqref{eq:estimate-after-mono} gives
\[
\begin{aligned}
    \frac{\lambda\alpha}{4}
    \delta_h^2
    \|\nu_{k,h}-\mu_h^\ast\|_1^2
    +
    \frac{1-\lambda L_mD_\Omega}{2}
    \|\nu_{k,h}-\mu_{k,h}\|_1^2
    \le
    D_{k,h}-D_{k+1,h}.
\end{aligned}
\]
Since
\[
    \|\mu_{k,h}-\mu_h^\ast\|_1^2
    \le
    2\|\nu_{k,h}-\mu_h^\ast\|_1^2
    +
    2\|\nu_{k,h}-\mu_{k,h}\|_1^2,
\]
we get
\begin{equation}\label{eq:finite-mesh-L1-descent}
    \gamma_h\|\mu_{k,h}-\mu_h^\ast\|_1^2
    \le
    D_{k,h}-D_{k+1,h}.
\end{equation}

Summing \eqref{eq:finite-mesh-L1-descent} from \(0\) to
\[
    \left\lceil
        \frac{256D_{0,h}}{\gamma_h (m_h^\ast)^4}
    \right\rceil
    -1
\]
gives an index \(j\) with
\[
    0\le j
    \le
    \left\lceil
        \frac{256D_{0,h}}{\gamma_h (m_h^\ast)^4}
    \right\rceil
\]
and
\[
    \|\mu_{j,h}-\mu_h^\ast\|_1^2
    \le
    \frac{(m_h^\ast)^4}{256}.
\]
Thus
\[
    \|\mu_{j,h}-\mu_h^\ast\|_1
    \le
    \frac{(m_h^\ast)^2}{16}
    \le
    \frac{m_h^\ast}{2}.
\]
Since \(\mu_{0,h}\) has full support, so does \(\mu_{j,h}\). The local
reverse-KL estimate from Lemma~\ref{lem:local-reverse-kl-finite-grid} gives
\[
    D_{j,h}
    \le
    2\|\mu_{j,h}-\mu_h^\ast\|_1
    \le
    \frac{(m_h^\ast)^2}{8}.
\]
Since \(D_{k,h}\) is nonincreasing, for every \(k\ge j\),
\[
    D_{k,h}\le D_{j,h}\le \frac{(m_h^\ast)^2}{8}.
\]
By Pinsker's inequality,
\[
    \|\mu_{k,h}-\mu_h^\ast\|_1
    \le
    \sqrt{2D_{k,h}}
    \le
    \frac{m_h^\ast}{2},
    \qquad k\ge j.
\]
Hence the local reverse-KL estimate applies for every \(k\ge j\), and
\[
    D_{k,h}\le 2\|\mu_{k,h}-\mu_h^\ast\|_1,
    \qquad k\ge j.
\]
Combining this with \eqref{eq:finite-mesh-L1-descent}, we obtain
\[
    D_{k+1,h}
    \le
    D_{k,h}
    -
    \frac{\gamma_h}{4}D_{k,h}^2,
    \qquad k\ge j.
\]
Thus, unless \(D_{k,h}=0\),
\[
    \frac1{D_{k+1,h}}
    \ge
    \frac1{D_{k,h}}+\frac{\gamma_h}{4}.
\]
Iterating from \(j\) and using
\[
    D_{j,h}^{-1}\ge \frac{8}{(m_h^\ast)^2}
\]
gives, with \(k_{0,h}:=j\),
\[
    D_{k,h}
    \le
    \frac{1}{
        \frac{8}{(m_h^\ast)^2}
        +
        \frac{\gamma_h}{4}(k-k_{0,h})
    },
    \qquad
    k\ge k_{0,h}.
\]
Pinsker's inequality and
\[
    W_1(\mu_{k,h},\mu_h^\ast)
    \le
    D_\Omega\|\mu_{k,h}-\mu_h^\ast\|_1
\]
give the stated \(W_1\)-estimate.

Now assume that \(\mu_h^\ast\) has full support on \(X_h\). Summing
\eqref{eq:finite-mesh-L1-descent} from \(0\) to
\[
    \left\lceil
        \frac{8D_{0,h}}{\gamma_h (m_h^\ast)^3}
    \right\rceil
    -1
\]
gives an index \(j\) with
\[
    0\le j
    \le
    \left\lceil
        \frac{8D_{0,h}}{\gamma_h (m_h^\ast)^3}
    \right\rceil
\]
and
\[
    \|\mu_{j,h}-\mu_h^\ast\|_1^2
    \le
    \frac{(m_h^\ast)^3}{8}.
\]
In particular,
\[
    \|\mu_{j,h}-\mu_h^\ast\|_1\le \frac{m_h^\ast}{2}.
\]
The local quadratic reverse-KL estimate from Lemma~\ref{lem:local-reverse-kl-finite-grid} gives
\[
    D_{j,h}
    \le
    \frac{1}{m_h^\ast}
    \|\mu_{j,h}-\mu_h^\ast\|_1^2
    \le
    \frac{(m_h^\ast)^2}{8}.
\]
As before, monotonicity of \(D_{k,h}\) and Pinsker's inequality imply
\[
    \|\mu_{k,h}-\mu_h^\ast\|_1\le \frac{m_h^\ast}{2},
    \qquad k\ge j.
\]
Hence the local quadratic reverse-KL estimate applies for all \(k\ge j\):
\[
    D_{k,h}
    \le
    \frac{1}{m_h^\ast}
    \|\mu_{k,h}-\mu_h^\ast\|_1^2.
\]
Combining this with \eqref{eq:finite-mesh-L1-descent} yields
\[
    D_{k+1,h}
    \le
    \left(1-\gamma_h m_h^\ast\right)D_{k,h},
    \qquad k\ge j.
\]
Thus, with \(k_{0,h}:=j\),
\[
    D_{k,h}
    \le
    \frac{(m_h^\ast)^2}{8}
    \left(1-\gamma_h m_h^\ast\right)^{k-k_{0,h}},
    \qquad
    k\ge k_{0,h}.
\]
Finally, Pinsker's inequality and
\(W_1\le D_\Omega\|\cdot\|_1\) give
\[
    W_1(\mu_{k,h},\mu_h^\ast)
    \le
    D_\Omega
    \frac{m_h^\ast}{2}
    \left(1-\gamma_h m_h^\ast\right)^{(k-k_{0,h})/2}.
\]

It remains to prove the bound on \(D_{0,h}\). Writing
\(\mu_h^\ast=\sum_i p_i^\ast\delta_{x_i}\) and
\(\mu_{0,h}=\sum_i p_{0,i}\delta_{x_i}\), we have
\[
\begin{aligned}
    D_{0,h}
    &=
    \sum_{i:p_i^\ast>0}
    p_i^\ast\log\frac{p_i^\ast}{p_{0,i}}  \\
    &=
    \sum_{i:p_i^\ast>0}p_i^\ast\log p_i^\ast
    -
    \sum_{i:p_i^\ast>0}p_i^\ast\log p_{0,i}  \\
    &\le
    -\sum_{i:p_i^\ast>0}p_i^\ast\log p_{0,i}  \\
    &\le
    -\log\left(
    \min_{x\in\supp\mu_h^\ast}\mu_{0,h}(x)
    \right).
\end{aligned}
\]
\end{proof}

\subsection{Proof of Proposition~\ref{prop:direct-mesh-stability}}
\begin{proof}
Let
\[
    \pi_h:=Q_{h\#}\mu^\ast\in\mathcal P(X_h).
\]
By the strong \(W_1\)-monotonicity assumption,
\[
\begin{aligned}
    \alpha W_1(\mu_h^\ast,\mu^\ast)^2
    &\le
    \langle J(\mu_h^\ast,\cdot)-J(\mu^\ast,\cdot),
    \mu_h^\ast-\mu^\ast\rangle \\
    &=
    \langle J(\mu_h^\ast,\cdot),\mu_h^\ast-\mu^\ast\rangle
    +
    \langle J(\mu^\ast,\cdot),\mu^\ast-\mu_h^\ast\rangle.
\end{aligned}
\]
Since \(\mu^\ast\) solves the continuous VI, using \(\eta=\mu_h^\ast\) gives
\[
    \langle J(\mu^\ast,\cdot),\mu_h^\ast-\mu^\ast\rangle\ge0.
\]
Therefore
\[
    \alpha W_1(\mu_h^\ast,\mu^\ast)^2
    \le
    \langle J(\mu_h^\ast,\cdot),\mu_h^\ast-\mu^\ast\rangle.
\]
Now decompose
\[
    \langle J(\mu_h^\ast,\cdot),\mu_h^\ast-\mu^\ast\rangle
    =
    \langle J(\mu_h^\ast,\cdot),\mu_h^\ast-\pi_h\rangle
    +
    \langle J(\mu_h^\ast,\cdot),\pi_h-\mu^\ast\rangle.
\]
Since \(\mu_h^\ast\) solves the finite-mesh VI and
\(\pi_h\in\mathcal P(X_h)\),
\[
    \langle J(\mu_h^\ast,\cdot),\pi_h-\mu_h^\ast\rangle\ge0.
\]
Thus
\[
    \alpha W_1(\mu_h^\ast,\mu^\ast)^2
    \le
    \langle J(\mu_h^\ast,\cdot),\pi_h-\mu^\ast\rangle.
\]
By the spatial Lipschitz bound and Kantorovich--Rubinstein duality,
\[
    \left|
    \langle J(\mu_h^\ast,\cdot),\pi_h-\mu^\ast\rangle
    \right|
    \le
    L_x W_1(\pi_h,\mu^\ast).
\]
Finally, since \(|Q_h(x)-x|\le h\),
\[
    W_1(\pi_h,\mu^\ast)
    =
    W_1(Q_{h\#}\mu^\ast,\mu^\ast)
    \le h.
\]
Therefore
\[
    \alpha W_1(\mu_h^\ast,\mu^\ast)^2
    \le
    L_x h,
\]
which proves the claim.
\end{proof}

\subsection{Proof of Theorem~\ref{thm:finite-iteration-continuous-error}}
\begin{proof}
For item \(1\), the triangle inequality gives
\[
    W_1(\mu_{k,h},\mu^\ast)
    \le
    W_1(\mu_{k,h},\mu_h^\ast)
    +
    W_1(\mu_h^\ast,\mu^\ast).
\]
The two estimates then follow directly from Theorem~\ref{thm:finite-mesh-last-iterate} and Proposition~\ref{prop:direct-mesh-stability}.

For item \(2\), summing \eqref{eq:finite-mesh-W1-descent} from \(0\) to
\(K-1\) gives
\begin{equation}\label{eq:W1-sum}
    \sum_{k=0}^{K-1}W_1(\mu_{k,h},\mu_h^\ast)^2
    \le
    \frac{D_{0,h}}{\beta}.
\end{equation}
Therefore
\[
    \min_{0\le k\le K-1}
    W_1(\mu_{k,h},\mu_h^\ast)
    \le
    \sqrt{
        \frac{D_{0,h}}{
            \beta K
        }
    }.
\]
Using the triangle inequality and Proposition~\ref{prop:direct-mesh-stability} proves the first bound in item \(2\). Under the displayed choices of \(h\) and \(K\), the optimization and mesh-approximation terms are each at most \(\varepsilon/2\), which proves the stated consequence.

For item \(3\), by convexity of \(W_1\),
\[
\begin{aligned}
    W_1(\bar\mu_{K,h},\mu_h^\ast)
    &=
    W_1\left(
        \frac1K\sum_{k=0}^{K-1}\mu_{k,h},
        \mu_h^\ast
    \right) \\
    &\le
    \frac1K\sum_{k=0}^{K-1}
    W_1(\mu_{k,h},\mu_h^\ast) \\
    &\le
    \left(
        \frac1K\sum_{k=0}^{K-1}
        W_1(\mu_{k,h},\mu_h^\ast)^2
    \right)^{1/2}.
\end{aligned}
\]
Using \eqref{eq:W1-sum}, we obtain
\[
    W_1(\bar\mu_{K,h},\mu_h^\ast)
    \le
    \sqrt{
        \frac{D_{0,h}}{
            \beta K
        }
    }.
\]
The triangle inequality and Proposition~\ref{prop:direct-mesh-stability} prove the first bound in item \(3\). Under the displayed choices of \(h\) and \(K\), each error term is at most \(\varepsilon/2\), which proves the stated consequence.
\end{proof}

\section{Proofs of Section~\ref{sec:MFE-selection}}
\begin{lemma}
\label{lem:entropy-consistency-mesh-projection}
For every \(\mu\in\mathcal P(\Omega)\),
\[
    Q_{h\#}\mu\Rightarrow \mu.
\]
Moreover,
\[
    D_{\rm KL}(Q_{h\#}\mu\|Q_{h\#}\rho)
    \to
    D_{\rm KL}(\mu\|\rho),
\]
with the convention that both sides may be \(+\infty\).
\end{lemma}

\begin{proof}
The weak convergence \(Q_{h\#}\mu\Rightarrow\mu\) follows from
\[
    |Q_h(x)-x|\le h
\]
and compactness of \(\Omega\).

For the entropy statement, the data-processing inequality gives
\[
    D_{\rm KL}(Q_{h\#}\mu\|Q_{h\#}\rho)
    \le
    D_{\rm KL}(\mu\|\rho).
\]
Hence
\[
    \limsup_{h\to0}
    D_{\rm KL}(Q_{h\#}\mu\|Q_{h\#}\rho)
    \le
    D_{\rm KL}(\mu\|\rho).
\]

For the reverse inequality, use the variational formula
\[
    D_{\rm KL}(\mu\|\rho)
    =
    \sup_{\varphi\in C(\Omega)}
    \left\{
        \int_\Omega \varphi\,d\mu
        -
        \log\int_\Omega e^\varphi\,d\rho
    \right\}.
\]
Fix \(\varphi\in C(\Omega)\). 
Then
\[
    \int_{X_h}\varphi\,d(Q_{h\#}\mu)
    =
    \int_\Omega \varphi(Q_h(x))\,\mu(dx)
    \to
    \int_\Omega \varphi(x)\,\mu(dx),
\]
and similarly
\[
    \int_{X_h}e^{\varphi}\,d(Q_{h\#}\rho)
    =
    \int_\Omega e^{\varphi(Q_h(x))}\,\rho(dx)
    \to
    \int_\Omega e^{\varphi(x)}\,\rho(dx).
\]
Therefore,
\[
\begin{aligned}
    \liminf_{h\to0}
    D_{\rm KL}(Q_{h\#}\mu\|Q_{h\#}\rho)
    &\ge
    \int_\Omega \varphi\,d\mu
    -
    \log\int_\Omega e^\varphi\,d\rho.
\end{aligned}
\]
Taking the supremum over \(\varphi\in C(\Omega)\) gives
\[
    \liminf_{h\to0}
    D_{\rm KL}(Q_{h\#}\mu\|Q_{h\#}\rho)
    \ge
    D_{\rm KL}(\mu\|\rho).
\]
Together with the limsup inequality, this proves the claim.
\end{proof}

\begin{lemma}
\label{lem:projected-equilibrium-approx-discrete}
Suppose Assumptions~\ref{ass:W1-lip} and~\ref{ass:spatial-lipschitz} hold.
Let \(\mu^\ast\in\mathcal P(\Omega)\) solve the VI:
\[
    \langle J(\mu^\ast,\cdot),\eta-\mu^\ast\rangle\ge0,
    \qquad
    \forall \eta\in\mathcal P(\Omega).
\]
Set
\[
    \pi_h:=Q_{h\#}\mu^\ast\in\mathcal P(X_h).
\]
Then for every \(\eta_h\in\mathcal P(X_h)\),
\[
    \langle J(\pi_h,\cdot),\eta_h-\pi_h\rangle
    \ge
    -(2L_m+L_x)h.
\]
\end{lemma}

\begin{proof}
Fix \(\eta_h\in\mathcal P(X_h)\). We have
\[
\begin{aligned}
    \langle J(\pi_h,\cdot),\eta_h-\pi_h\rangle
    =
    \langle J(\mu^\ast,\cdot),\eta_h-\mu^\ast\rangle
    +
    \langle J(\pi_h,\cdot)-J(\mu^\ast,\cdot),\eta_h-\pi_h\rangle
    +
    \langle J(\mu^\ast,\cdot),\mu^\ast-\pi_h\rangle.
\end{aligned}
\]
The first term is nonnegative by the VI. The second term satisfies
\[
    \left|
    \langle J(\pi_h,\cdot)-J(\mu^\ast,\cdot),\eta_h-\pi_h\rangle
    \right|
    \le
    2\|J(\pi_h,\cdot)-J(\mu^\ast,\cdot)\|_\infty
    \le
    2L_m W_1(\pi_h,\mu^\ast).
\]
Since \(\pi_h=Q_{h\#}\mu^\ast\) and \(|Q_h(x)-x|\le h\),
\[
    W_1(\pi_h,\mu^\ast)\le h.
\]
For the third term, the spatial Lipschitz bound gives
\[
    \left|
    \langle J(\mu^\ast,\cdot),\mu^\ast-\pi_h\rangle
    \right|
    \le
    L_x W_1(\mu^\ast,\pi_h)
    \le
    L_x h.
\]
Combining the estimates gives
\[
    \langle J(\pi_h,\cdot),\eta_h-\pi_h\rangle
    \ge
    -(2L_m+L_x)h.
\]
\end{proof}

\begin{lemma}
\label{lem:composite-euler-equivalence}
Let \(X_h\) be finite and \(R_h:\mathcal P(X_h)\to \mathbb R\cup\{+\infty\}\) be convex. Suppose that
\(\mu_h^\varepsilon\in\mathcal P(X_h)\) lies in the relative interior of
\(\mathcal P(X_h)\), and that \(R_h\) is differentiable at
\(\mu_h^\varepsilon\) in all feasible directions. Denote its first variation
by
\[
    \frac{\delta R_h}{\delta \mu}
    (\mu_h^\varepsilon,x),
    \qquad x\in X_h.
\]
Then the following two conditions are equivalent.

\begin{enumerate}
\item The composite Tikhonov VI:
\begin{equation}
    \left\langle
        J(\mu_h^\varepsilon,\cdot),
        \eta_h-\mu_h^\varepsilon
    \right\rangle
    +
    \varepsilon
    \left[
        R_h(\eta_h)-R_h(\mu_h^\varepsilon)
    \right]
    \ge 0,
    \qquad
    \forall \eta_h\in\mathcal P(X_h).
\label{eq:composite-vi}
\end{equation}

\item The Euler, or first-order, VI:
\begin{equation}
    \left\langle
        J(\mu_h^\varepsilon,\cdot)
        +
        \varepsilon
        \frac{\delta R_h}{\delta \mu}
        (\mu_h^\varepsilon,\cdot),
        \eta_h-\mu_h^\varepsilon
    \right\rangle
    \ge0,
    \qquad
    \forall \eta_h\in\mathcal P(X_h).
\label{eq:euler-vi}
\end{equation}
\end{enumerate}
Here the first variation is understood modulo additive constants, since
\[
    \langle c,\eta_h-\mu_h^\varepsilon\rangle=0
    \qquad
    \forall c\in\mathbb R,\quad
    \forall \eta_h\in\mathcal P(X_h).
\]
\end{lemma}

\begin{proof}
We first prove that \eqref{eq:composite-vi} implies
\eqref{eq:euler-vi}. Fix \(\eta_h\in\mathcal P(X_h)\) and set
\[
    \eta_{h,t}:=(1-t)\mu_h^\varepsilon+t\eta_h,
    \qquad t\in(0,1).
\]
Since \(\mathcal P(X_h)\) is convex, \(\eta_{h,t}\in\mathcal P(X_h)\). Applying
\eqref{eq:composite-vi} with \(\eta_{h,t}\) gives
\[
    \left\langle
        J(\mu_h^\varepsilon,\cdot),
        \eta_{h,t}-\mu_h^\varepsilon
    \right\rangle
    +
    \varepsilon
    \left[
        R_h(\eta_{h,t})-R_h(\mu_h^\varepsilon)
    \right]
    \ge0.
\]
Since
\[
    \eta_{h,t}-\mu_h^\varepsilon
    =
    t(\eta_h-\mu_h^\varepsilon),
\]
we have
\[
    t
    \left\langle
        J(\mu_h^\varepsilon,\cdot),
        \eta_h-\mu_h^\varepsilon
    \right\rangle
    +
    \varepsilon
    \left[
        R_h(\mu_h^\varepsilon+t(\eta_h-\mu_h^\varepsilon))
        -
        R_h(\mu_h^\varepsilon)
    \right]
    \ge0.
\]
Dividing by \(t>0\), we obtain
\[
    \left\langle
        J(\mu_h^\varepsilon,\cdot),
        \eta_h-\mu_h^\varepsilon
    \right\rangle
    +
    \varepsilon
    \frac{
        R_h(\mu_h^\varepsilon+t(\eta_h-\mu_h^\varepsilon))
        -
        R_h(\mu_h^\varepsilon)
    }{t}
    \ge0.
\]
Letting \(t\downarrow0\) and using differentiability of \(R_h\) at
\(\mu_h^\varepsilon\), we get
\[
    \left\langle
        J(\mu_h^\varepsilon,\cdot),
        \eta_h-\mu_h^\varepsilon
    \right\rangle
    +
    \varepsilon
    \left\langle
        \frac{\delta R_h}{\delta\mu}
        (\mu_h^\varepsilon,\cdot),
        \eta_h-\mu_h^\varepsilon
    \right\rangle
    \ge0.
\]
This is exactly \eqref{eq:euler-vi}.

Conversely, assume \eqref{eq:euler-vi}. By convexity of \(R_h\), for every
\(\eta_h\in\mathcal P(X_h)\),
\[
    R_h(\eta_h)-R_h(\mu_h^\varepsilon)
    \ge
    \left\langle
        \frac{\delta R_h}{\delta\mu}
        (\mu_h^\varepsilon,\cdot),
        \eta_h-\mu_h^\varepsilon
    \right\rangle.
\]
Therefore,
\[
\begin{aligned}
    \left\langle
        J(\mu_h^\varepsilon,\cdot),
        \eta_h-\mu_h^\varepsilon
    \right\rangle
    +
    \varepsilon
    \left[
        R_h(\eta_h)-R_h(\mu_h^\varepsilon)
    \right]
    \ge
    \left\langle
        J(\mu_h^\varepsilon,\cdot)
        +
        \varepsilon
        \frac{\delta R_h}{\delta\mu}
        (\mu_h^\varepsilon,\cdot),
        \eta_h-\mu_h^\varepsilon
    \right\rangle
    \ge0.
\end{aligned}
\]
Hence \eqref{eq:composite-vi} holds. This proves the equivalence.
\end{proof}

In our setting,
\(
    R_h(\mu_h):=D_{\rm KL}(\mu_h\|\rho_h),
\)
and \(\rho_h\) has full support on \(X_h\). The composite regularized VI reads
\[
    \left\langle
        J(\mu_h^\varepsilon,\cdot),
        \eta_h-\mu_h^\varepsilon
    \right\rangle
    +
    \varepsilon
    \left[
        D_{\rm KL}(\eta_h\|\rho_h)
        -
        D_{\rm KL}(\mu_h^\varepsilon\|\rho_h)
    \right]
    \ge0,
    \qquad
    \forall \eta_h\in\mathcal P(X_h).
\]
Equivalently, 
\[
    \mu_h^\varepsilon
    \in
    \operatorname*{argmin}_{\eta_h\in\mathcal P(X_h)}
    \left\{
        \left\langle
            J(\mu_h^\varepsilon,\cdot),
            \eta_h
        \right\rangle
        +
        \varepsilon D_{\rm KL}(\eta_h\|\rho_h)
    \right\}.
\]
The minimizer of this strictly convex problem is explicitly given by the
Gibbs formula
\[
    \mu_h^\varepsilon(x)
    =
    \frac{
        \rho_h(x)\exp\left(-J(\mu_h^\varepsilon,x)/\varepsilon\right)
    }{
        \sum_{y\in X_h}
        \rho_h(y)\exp\left(-J(\mu_h^\varepsilon,y)/\varepsilon\right)
    },
    \qquad x\in X_h.
\]
In particular,
\[
    \mu_h^\varepsilon(x)>0,
    \qquad x\in X_h,
\]
so \(\mu_h^\varepsilon\in\mathcal P(X_h)\) lies in the relative interior of
\(\mathcal P(X_h)\), and the KL functional $R_h$ is differentiable at \(\mu_h^\varepsilon\). Its first
variation is
\[
    \frac{\delta R_h}{\delta\mu}
    (\mu_h^\varepsilon,x)
    =
    \log\frac{\mu_h^\varepsilon(x)}{\rho_h(x)}+1.
\]
Since \(\eta_h-\mu_h^\varepsilon\) has total mass zero, the additive constant
\(1\) does not contribute. Therefore the composite KL-regularized VI is
equivalent to the Euler VI
\[
    \left\langle
        J(\mu_h^\varepsilon,\cdot)
        +
        \varepsilon
        \log\frac{\mu_h^\varepsilon}{\rho_h},
        \eta_h-\mu_h^\varepsilon
    \right\rangle
    \ge0,
    \qquad
    \forall \eta_h\in\mathcal P(X_h).
\]

\begin{proof}[{\bf Proof of Theorem~\ref{thm:joint-mesh-tikhonov-selection}}]
We first prove a selection estimate. By Lemma~\ref{lem:composite-euler-equivalence}, \(\mu_h^{\varepsilon_h}\in\mathcal P(X_h)\) satisfies
\[
    \langle J(\mu_h^{\varepsilon_h},\cdot),\eta_h-\mu_h^{\varepsilon_h}\rangle
    +
    \varepsilon_h
    \left[
        R_h(\eta_h)-R_h(\mu_h^{\varepsilon_h})
    \right]
    \ge0,
    \qquad
    \forall \eta_h\in\mathcal P(X_h).
\]
In particular, setting \(\eta_h=\pi_h:=Q_{h\#}\mu^\dagger\) gives
\[
    \langle J(\mu_h^{\varepsilon_h},\cdot),
    \pi_h-\mu_h^{\varepsilon_h}\rangle
    +
    \varepsilon_h
    \left[
        R_h(\pi_h)-R_h(\mu_h^{\varepsilon_h})
    \right]
    \ge0.
\]
By monotonicity and Lemma~\ref{lem:projected-equilibrium-approx-discrete},
\[
\begin{aligned}
    \langle J(\mu_h^{\varepsilon_h},\cdot),
    \mu_h^{\varepsilon_h}-\pi_h\rangle
    =
    \langle J(\mu_h^{\varepsilon_h},\cdot)-J(\pi_h,\cdot),
    \mu_h^{\varepsilon_h}-\pi_h\rangle
    +
    \langle J(\pi_h,\cdot),
    \mu_h^{\varepsilon_h}-\pi_h\rangle
    \ge
    -(2L_m+L_x)h.
\end{aligned}
\]
Therefore
\[
    R_h(\mu_h^{\varepsilon_h})
    \le
    R_h(\pi_h)+\frac{(2L_m+L_x)h}{\varepsilon_h}.
\]
By Lemma~\ref{lem:entropy-consistency-mesh-projection},
\[
    R_h(\pi_h)
    =
    D_{\rm KL}(Q_{h\#}\mu^\dagger\|Q_{h\#}\rho)
    \to
    D_{\rm KL}(\mu^\dagger\|\rho)
    =
    R(\mu^\dagger).
\]
Since \(h/\varepsilon_h\to0\), we have
\[
    \limsup_{h\to0}R_h(\mu_h^{\varepsilon_h})
    \le
    R(\mu^\dagger).
\]

By compactness of \(\mathcal P(\Omega)\), every sequence
\(\mu_h^{\varepsilon_h}\) has weakly convergent subsequences. Let
\[
    \mu_h^{\varepsilon_h}\Rightarrow \bar\mu
\]
along such a subsequence. We show that \(\bar\mu\in\mathcal S\). 

Let
\(\eta\in\mathcal P(\Omega)\), and set
\[
    \eta_h:=Q_{h\#}\eta.
\]
The regularized VI gives
\[
    \langle J(\mu_h^{\varepsilon_h},\cdot),
    \eta_h-\mu_h^{\varepsilon_h}\rangle
    \ge
    \varepsilon_h
    \left[
        R_h(\mu_h^{\varepsilon_h})-R_h(\eta_h)
    \right].
\]
We have
\[
R_h(\mu_h^{\varepsilon_h})\ge0,\qquad
    R_h(\eta_h)
    \le
    \log\frac1{\rho_{\min,h}}.
\]
The balance condition
\[
    \varepsilon_h\log\frac1{\rho_{\min,h}}\to0
\]
therefore yields
\[
    \liminf_{h\to0}
    \langle J(\mu_h^{\varepsilon_h},\cdot),
    \eta_h-\mu_h^{\varepsilon_h}\rangle
    \ge0.
\]
Using weak convergence, the fact that \(\eta_h\Rightarrow\eta\), and the
uniform convergence implied by Assumption~\ref{ass:W1-lip}, we pass to the
limit and obtain
\[
    \langle J(\bar\mu,\cdot),\eta-\bar\mu\rangle\ge0.
\]
Since \(\eta\in\mathcal P(\Omega)\) was arbitrary, \(\bar\mu\in\mathcal S\).

Finally, by the joint lower semicontinuity of relative entropy,
\[
    R(\bar\mu)
    =
    D_{\rm KL}(\bar\mu\|\rho)
    \le
    \liminf_{h\to0}
    D_{\rm KL}(\mu_h^{\varepsilon_h}\|\rho_h)
    =
    \liminf_{h\to0}R_h(\mu_h^{\varepsilon_h}).
\]
Combining this with the previous limsup estimate gives
\[
    R(\bar\mu)\le R(\mu^\dagger).
\]
Since \(\bar\mu\in\mathcal S\) and \(\mu^\dagger\) is the unique minimizer of
\(R\) over \(\mathcal S\), we conclude that
\[
    \bar\mu=\mu^\dagger.
\]
Every subsequential limit is therefore \(\mu^\dagger\), and hence
\[
    \mu_h^{\varepsilon_h}\Rightarrow \mu^\dagger.
\]
\end{proof}
\end{document}